\documentclass[11pt, reqno]{amsart}
\usepackage[a4paper,vmargin={3.5cm,3.5cm},hmargin={2.5cm,2.5cm},bottom=30mm]{geometry}
\usepackage{amsfonts}
\usepackage{amsthm,amsmath,mathtools}
\usepackage{graphicx}
\usepackage{amsbsy}
\usepackage{bm}

\usepackage{amssymb} 
\usepackage{enumerate,enumitem}
\usepackage{color}

\DeclarePairedDelimiter{\ceil}{\lceil}{\rceil}
\DeclarePairedDelimiter{\floor}{\lfloor}{\rfloor}

\usepackage{pgf,tikz}
\usetikzlibrary{arrows}

\usepackage{units}
\usepackage{enumerate}
\usepackage{bbm}
\usepackage{exscale}
\usepackage{hyperref}
\hypersetup{ colorlinks=true, linkcolor=blue, citecolor=blue,
  filecolor=blue, urlcolor=blue}
  \numberwithin{equation}{section} 
  \date{}

\DeclareMathOperator{\one}{\mathbbm{1}} 

\usepackage[comma, sort&compress]{natbib}

\theoremstyle{plain} 
    \newtheorem{theorem}{Theorem}
    \numberwithin{theorem}{section} 
    \newtheorem{lemma}[theorem]{Lemma}
    \newtheorem{proposition}[theorem]{Proposition}
    \newtheorem{corollary}[theorem]{Corollary}
    \newtheorem{claim}[theorem]{Claim}

\theoremstyle{definition} 
    
    \newtheorem{fact}[theorem]{Fact}
    
    \newtheorem{remark}[theorem]{Remark}

\DeclareMathOperator{\R}{\mathbb{R}}

\DeclareMathOperator{\Z}{\mathbb{Z}}
\DeclareMathOperator{\N}{\mathbb{N}}

\DeclareMathOperator{\De}{d}

\newcommand{\e}{\mathrm{e}}

\newcommand{\f}{\frac}  

\newcommand{\vr}{\varphi}
\newcommand{\var}{\mathbf{Var}}

\newcommand{\prob}{\mathbf{P}}
\newcommand{\E}{\mathbf{E}}

\newcommand{\la}{\left\langle}
\newcommand{\ra}{\right\rangle}

 \makeatletter
 \def\paragraph{\@startsection{paragraph}{4}%
 \z@\z@{-\fontdimen2\font}%
   {\normalfont\itshape}}\makeatother

\title[Scaling limit of the membrane model]{\vspace{-2cm}\textbf{The scaling limit of the membrane model} \vspace{0.5cm}}

\author[A. Cipriani]{Alessandra Cipriani}
\address{TU Delft (DIAM), Building 28, van Mourik Broekmanweg 6, 2628 XE, Delft, The Netherlands}
\email{A.Cipriani@tudelft.nl}
\author[B. Dan]{Biltu Dan}
\author[R.~S.~Hazra]{Rajat Subhra Hazra}
\address{ISI Kolkata, 203, B.T. Road, Kolkata, 700108, India}
\email{biltudanmath@gmail.com, rajatmaths@gmail.com}

\date\today
\begin{document}
\begin{abstract}
On the integer lattice we consider the discrete membrane model, a random interface 
in which the field has Laplacian interaction. We prove that, under appropriate rescaling, 
the discrete membrane model converges to the continuum membrane model in $d\ge 2$. Namely, it 
is shown that the scaling limit in $d=2,\,3$ is a H\"older continuous random field, while in $d\ge 4$ the membrane model converges 
to a random distribution. As a by-product of the proof in $d=2,\,3$, we obtain the scaling limit of the maximum. 
This work complements the analogous results of \cite{CarJDScaling} in $d=1$.
\end{abstract}

\maketitle
\section{Introduction}

The main object of study in this article is the membrane model (MM), also known as discrete bilaplacian model. 
The membrane model is a special instance of a more general class of interface models 
in which the interaction of the system is governed by the exponential of an Hamiltonian 
function $H:\mathbb{R}^{\mathbb{Z}^{d}}\rightarrow [0,\infty)$. More specifically, random interfaces are fields 
$\varphi=(\varphi_{x})_{x\in\Z^{d}}$, whose distribution
is determined by the probability measure on $\mathbb{R}^{\mathbb{Z}^{d}}$,
$d\ge1$, with density
\[
\prob_{W}(\mathrm{d}\varphi):=\frac{\mathrm{e}^{-H(\varphi)}}{Z_W}\prod_{x\in{W}}\mathrm{d}\varphi_{x}\prod_{x\in\mathbb{Z}^{d}\setminus W}\delta_{0}(\mathrm{d}\varphi_{x}),
\]
where $W\Subset\mathbb{Z}^{d}$ is a finite subset, $\mathrm{d}\varphi_{x}$ is the 
1-dimensional Lebesgue measure on $\R$, $\delta_{0}$ is the Dirac measure at $0,$ and $Z_{W}$ is a normalising constant. 
We are imposing zero boundary conditions i.e. almost surely $\varphi_{x}=0$
for all $x\in\mathbb{Z}^{d}\setminus{W}$, but the definition
holds for more general boundary conditions. A relevant example is where the Hamiltonian is driven by a convex function of the gradient, that is, 
$H(\varphi)= \sum_{x\sim y} V(\varphi_y- \varphi_x)$, $V:\R\to\R$ convex, and the sum being over nearest neighbours. 
The most well-known among these interfaces is the discrete Gaussian free field (DGFF) when $V(x)\propto x^2/2$. The quadratic potential 
allows one to have 
various tools at one's disposal, like the random walk representation of covariances and inequalities like FKG. These tools can be generalised to 
(strictly) convex potentials in the form of the Brascamp--Lieb inequality and the 
Helffer--Sj\"ostrand random walk representation. We refer to \cite{NaddafSpencer,giacomin2001, Funaki, velenikloc} for an overview. 
Outside the convex regime, the non-convex regime was recently studied for example 
in~\cite{cotar2009strict, biskup:spohn}. 

A very natural probabilistic question one can ask oneself is: ``What happens to a random interface 
when one rescales it suitably?''. In $d=1$ in the example of the DGFF the scaling limit is the Brownian bridge. 
In $d\ge 2$ the limit, the continuum Gaussian free field,
is not a random variable and can only be interpreted in the language of distribution theory 
(see for example \cite{Sheff, biskup:survey}). 
The importance of the continuum Gaussian free field in $d=2$ relies on its universality property 
due to conformal invariance, and links it to other stochastic processes like SLE, CLE, 
and Liouville quantum gravity. 
The recent developments concerning the extreme 
value theory of DGFF (and, more generally, log-correlated fields) have shown 
impressive connections also to number theory, branching processes and random matrices.

In comparison to the DGFF, the membrane model has received slightly less attention, 
mainly due to the technical challenges intrinsic of the model. 
It is the Gaussian interface for which
\begin{equation}
H(\varphi):=\frac{1}{2}\sum_{x\in\mathbb{Z}^{d}}|\Delta_1\varphi_{x}|^{2}\label{eq:ham}
\end{equation}
and $\Delta_1$ is the discrete Laplacian defined by 
$$\Delta_1 f(x)= \frac1{2d}\sum_{y\sim x} (f(y)-f(x)),\quad f:\Z^d\to \R, \;x\in\Z^d.$$ 
In case $W=V_N:=[-N,\,N]^d\cap \Z^d$, we will denote the measure $\prob_{V_N}$ with Hamiltonian~\eqref{eq:ham} by $\prob_N$. 
 Introduced by \cite{Sakagawa} in the probabilistic 
literature, the MM looks for certain aspects very similar to the DGFF: 
it is log-correlated in $d=4$, has a supercritical regime in $d\ge 5$ and 
is subcritical in $d\le 3$. In particular in $d=2,\,3,\,4$ there is no thermodynamic limit of the measures $\prob_N$ as $N\uparrow \infty$. The MM displays however certain crucial difficulties, in 
that for example it exhibits no random walk representation, and several 
correlation inequalities are lacking. Nonetheless it is possible, via analytic and numerical methods, to 
obtain sharp results on its behaviour. 
Examples are the study of the entropic repulsion and pinning effects \citep{CaravennaDeuschel_pin, Kurt_d4, Kurt_d5, BCK17, AKW16}, extreme value theory \citep{CCHInterfaces}, 
and connections to other statistical mechanics models \citep{CHR}. 
In this framework we present our work which aims at determining the scaling limit of the bilaplacian model. 
The answer in $d=1$ was given by \cite{CarJDScaling}, who also look at the situation in which a pinning force is added to the 
model. We complement their work by determining the scaling limit in all $d\ge 2$. We also mention~\cite{HrVe}, 
who consider general semiflexible membranes as well with a different scaling approach. 
Their results are derived using an integrated random walk representation which is difficult to adapt in higher dimensions.  
\begin{figure}[ht!]
\includegraphics[scale=0.5]{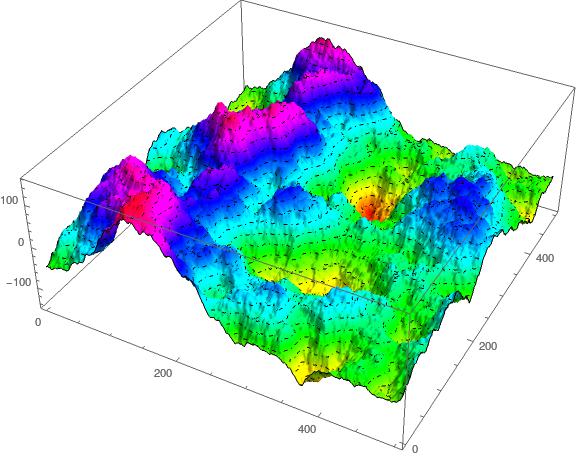}\caption{A sample of the MM in $d=2$ on a box of side-length $500$.}
\end{figure}
The main contributions of this article are as follows: 
\begin{itemize}[label=$\spadesuit$]
\item in $d=2,\,3$ we consider the discrete membrane model 
on a box of side-length $2N$ and interpolate it in a continuous way. We show that the 
process converges to a real-valued process with continuous trajectories and the convergence takes place in the space of 
continuous functions (see Theorem~\ref{thm:low_d}). The utility of this type of convergence is that it yields the scaling limit of the discrete maximum exploiting the 
continuous mapping theorem (Corollary~\ref{cor:max}). While the limiting maximum of the discrete membrane model in $d\ge 5$ was 
derived by~\cite{CCHInterfaces}, in $d=4$ the problem remains open as far as the authors know (tightness can be derived 
from~\cite{ding:roy:ofer}). The limit field also turns to be H\"older continuous with exponent less than $1$ in $d=2$ and less than $1/2$ in $d=3$.

The proof of the above facts is based on two basic steps: tightness and finite dimensional convergence. Tightness 
depends on the gradient estimates of the discrete Green's functions which were very recently derived in~\cite{Mueller:Sch:2017}; 
finite dimensional convergence follows from the convergence of the Green's function.

\item In $d\ge 4$ the limiting process on a sufficiently nice domain $D$ will be a fractional Gaussian field with Hurst parameter $H:=s-d/2$ on $D$. The theory 
of fractional Gaussian fields was surveyed recently in~\cite{LSSW}. The authors there construct the continuum membrane model 
using characteristic functionals. We take here a bit different route and give a representation using the eigenvalues of the 
biharmonic operator in the continuum. 
We remark however that these eigenvalues differ from the square of the Laplacian eigenvalues 
due to boundary conditions. The GFF theory which is based on $H^1_0(D)$ (the 
first order Sobolev space) needs to be replaced by $H^2_0(D)$ (second order Sobolev space). 

Our main result is given in Theorem~\ref{thm:critical_d}. Its proof is again split into two steps: 
finite dimensional convergence and tightness. 
Both steps crucially require an approximation result of PDEs given by \cite{thomee}: there 
he gives quantitative estimates on the 
approximation of solutions of PDEs involving ``nice" elliptic operators by their 
discrete counterparts.
We believe that the techniques used in that article might have implications 
in the development of 
the theory of the membrane model, in particular the idea of tackling boundary values 
by rescaling the standard 
discrete Sobolev norm around the boundary. Especially in $d=4$ this allows one to overcome the difficulty of extending estimates from the bulk 
up to the boundary, which is generally one stumbling block in the study of the MM. 
\item In $d\ge 5$ we also consider the infinite volume membrane model on $\Z^d$.
We show in Lemma~\ref{lemma:psi} that the limit is the fractional Gaussian field of Hurst parameter $H:=2-d/2<0$ on $\R^d$ (see~\cite{LSSW}) and we prove in Theorem~\ref{thm:big_d} the convergence with the help of characteristic functionals. 
We utilise the classical result of \cite{Fernique:1968} (recently extended in the tempered distribution setting by \cite{BOY2017}) 
stating that convergence of tempered distributions is equivalent to that of their characteristic functionals. Technical tools useful for this scope 
are the explicit Fourier transform of the infinite volume Green's function and the Poisson summation formula.
\end{itemize}
We stress that, regardless of the dimension, the field is always rescaled as $N^{(d-4)/2}\vr_{Nx}$ for $x\in N^{-1} \Z^d$. Heuristically, the factor $N^{4-d}$ corresponds to the order of growth of the variance of the model in a box, which we recall here for completeness.
\begin{enumerate}[leftmargin=*,label=\roman*)]
 \item In $d=2,\,3$ if $d(\cdot)$ denotes the distance to the boundary of $V_N$ one has for some constant $C>0$ \cite[Theorem~1.1]{Mueller:Sch:2017}
\[
 |\mathbf{Cov}_N(\vr_x,\,\vr_y)|\le C\min\left(d(x)^{2-d/2}d(y)^{2-d/2},\,\frac{d(x)^2 d(y)^2}{(\|x-y\|+1)^2}\right).
\]
\item In $d=4$ let us denote the bulk of $V_N$ by  $V_N^{\delta}:= \{x\in V_N: d(x)>\delta N\}$ for $\delta\in (0,1)$. Then from \citet[Lemma~2.1]{Cip13} we have: there exists a constant $C(\delta)>0$ such that
 \[
\sup_{x,y\in V_N^{\delta}}\Big|\mathbf{Cov}_N(\vr_x,\,\vr_y)- \frac{8}{\pi^2}(\log N-\log(\|x-y\|+1)) \Big|\le C(\delta).
\]
Asymptotics up to the boundary are not known to the best of the authors' knowledge. The approach of \cite{thomee} allows to circumvent this lack of estimates. 
\item In $d\ge 5$ the infinite volume covariance satisfies \cite[Lemma 5.1]{Sakagawa}
\[|\mathbf{Cov}(\vr_x,\,\vr_y)| \sim C_d \|x-y\|^{4-d} \text{ as $\|x-y\|\to\infty$.}\]
\end{enumerate}
Interestingly this reflects the behavior of the characteristic singular solution (fundamental solution) of the biharmonic equation, which is
\[
 \begin{cases}
  C_d \|x\|^{4-d}& d\text{ odd or }d\text{ even and } d\ge 6\\
  C_d \|x\|^{4-d}\log\|x\|&d\text{ even and }d\le 4.
 \end{cases}
\]
The reader can consult~\cite{MayborodaReg},~\citet[Section~5]{MitreaMitrea} and references therein for sharp pointwise estimates of the Green's function of the bilaplacian in general domains and for regularity properties of the biharmonic Green's function. 

We would like to conclude the Introduction with a few open questions:
\begin{itemize}
\item Is the maximum of the discrete membrane model at the critical dimension scaling to a randomly shifted Gumbel, as predicted by \cite{ding:roy:ofer}?
\item What will the scaling limit be for interfaces with mixed Hamiltonian of the form 
$H(\varphi):=\sum_x V_1(\nabla\varphi_x)+\sum_x V_2(\Delta\varphi_x)$, $V_1,\,V_2$ convex functions (in particular, $V_1\equiv 0$)? Results on these models 
were shown in \citet{Caravenna/Borecki:2010} in $d=1$. 
\end{itemize}
\paragraph{Structure of the paper}  In Section~\ref{sec:d<4} we handle the case $d\in\{2,\,3\}$, while in Section~\ref{sec:limit_des} we treat the finite-volume 
case in $d\ge 4$. In Section~\ref{sec:big_d} we analyse the case of the infinite-volume model in $d\ge 5$. To keep the article self-contained 
in Appendix~\ref{appendix:T} we discuss the results from \cite{thomee} and also deduce a quantitative 
version of the approximation result proved there.

\paragraph{Acknowledgements} The first author is supported by the grant 613.009.102 of the Netherlands Organisation for Scientific Research (NWO) and was supported by the EPSRC grant EP/N004566/1 and the 
Dutch stochastics cluster STAR (Stochastics -- Theoretical and Applied Research) while affiliated with the University of Bath. The 
first and third author acknowledge the MFO grant RiP 1706s. The third author also thanks the NETWORKS grant in the Netherlands 
and the University of Leiden where a part of the work was carried out. 
All authors are very grateful to Vidar Thom\'ee who kindly provided the paper~\citet{thomee}. Thanks also to Francesco Caravenna and Noemi Kurt for 
helpful discussions, and to 
Stefan M\"uller and Florian Schweiger for sharing their article \cite{Mueller:Sch:2017}. F. Schweiger also 
observed that Theorem~\ref{thm:tightness0} yields global H\"older continuity, and also how to improve the H\"older exponent in $d=2$ (Lemma~\ref{moment_bound0}). Finally it is our pleasure to thank an anonymous referee for his/her insightful comments and careful reading which improved the quality and readability of the paper.


%
\subsection*{Notation}We fix a constant $\kappa:=(2d)^{-1}$ throughout the whole paper. In the following $C>0$ always denotes a universal 
constant whose value however may change in each occurence. 
We will use $\stackrel{d}{\to}$ to denote 
convergence in distribution. We denote, for 
any $y=(y_1,\ldots,\,y_d)\in\R^d$, $d\ge 1$, the ``integer part'' of $y$ as $\lfloor y\rfloor=(\lfloor y_1\rfloor,\ldots,\,\lfloor y_d\rfloor)$ and 
similarly $\{y\}=y-\lfloor y\rfloor$ is the ``fractional part'' of $y$. 

\section{Convergence in \texorpdfstring{$d=2,\,3$}{}}\label{sec:d<4}
\subsection{Description of the limiting field}
Let $V=(-1, 1)^d$ and $V_{N}= N\overline V \cap \Z^d$, where $N\in\N$. Let $(\varphi_x)_{x\in V_{N-1}}$ be the MM 
on $V_{N-1}$ and let $G_{N-1}$ be the covariance function for this model. 
It is known~\cite[Section~1]{Kurt_thesis} to satisfy the following discrete boundary value problem for all $x\in V_{N-1}$:
\[
 \begin{cases}
\Delta_1^2 G_{N-1}(x,\,y)=\delta_x(y),&y\in V_{N-1}\\
G_{N-1}(x,\,y)=0, & y\notin V_{N-1}
 \end{cases}.
\]
First we want to define a {\it continuous} interpolation $\Psi_N$ of the discrete field to have convergence in the space of continuous functions. 
There are many ways to define the field $\left( \Psi_{N}(t)\right)_{t\in\overline V}$. 
We take one of the simplest geometric ways which is akin to the interpolation of simple random walk 
trajectories in Donsker's invariance principle. Mind that we take the domain as a square since the recent gradient estimates and convergence of the Green's function of \cite{Mueller:Sch:2017} can be applied easily.
\paragraph{Interpolation in $d=2$.} Let $t=(t_1\,,\,t_2)\in\overline V$. Then $p :=Nt$ lies in the square box with vertices ${ a}=\lfloor Nt\rfloor, {b}= \lfloor Nt\rfloor +e_1, { c}= \lfloor Nt\rfloor +e_1+e_2, { d}= \lfloor Nt\rfloor + e_2$, where $e_1,\,e_2$ are the standard basis vectors of $\R^2$. Suppose ${p}$ is a point in the triangle ${abc}$. Then we can write ${p}=\alpha { a}+\beta {b}+\gamma { c}$ with $ \alpha = 1-\{Nt_1\},\,\beta = \{Nt_1\}-\{Nt_2\},\,\gamma = \{Nt_2\}.$ And in this case we define 
$$\Psi_{N}(t)=\frac\kappa N[ \alpha\vr_{\lfloor Nt\rfloor} + \beta \vr_{\lfloor Nt\rfloor+e_1} + \gamma\vr_{\lfloor Nt\rfloor+e_1+e_2}].$$ Similarly, if $p\in\triangle acd$ then we define
$$\Psi_{N}(t)= \frac\kappa N[\alpha'\vr_{\lfloor Nt\rfloor} + \beta' \vr_{\lfloor Nt\rfloor+e_2} + \gamma'\vr_{\lfloor Nt\rfloor+e_1+e_2}]$$
where $$ \alpha' = 1-\{Nt_2\},\,\beta' = \{Nt_2\}-\{Nt_1\},\,\gamma' = \{Nt_1\}.$$
Thus the interpolated field $\left( \Psi_N(t)\right)_{t\in\overline V}$ is defined by

\begin{align*}
\Psi_{N}(t)&=\frac{\kappa}{N}[\vr_{\floor {Nt}}+\{Nt_i\}\left(\vr_{\floor {Nt}+e_i}-\vr_{\floor {Nt}}\right)\\
&+\{Nt_j\}\left(\vr_{\floor {Nt}+e_i+e_j}-\vr_{\floor {Nt}+e_i}\right)] \,,\quad\text{if}\,\{Nt_i\}\ge\{Nt_j\}
\end{align*}
where $i,\,j\in\{1,\,2\}$, $i\neq j$.

\paragraph{Interpolation in $d=3$.} In $d=3$ the interpolated field can be defined in the same way as above. We use tetrahedrons to define the interpolated field as
\begin{align*}
\Psi_{N}(t)&=\frac{\kappa}{\sqrt N}[\vr_{\floor {Nt}}+\{Nt_i\}\left(\vr_{\floor {Nt}+e_i}-\vr_{\floor {Nt}}\right)\\
&+\{Nt_j\}\left(\vr_{\floor {Nt}+e_i+e_j}-\vr_{\floor {Nt}+e_i}\right)\\
&+\{Nt_k\}\left(\vr_{\floor {Nt}+e_i+e_j+e_k}-\vr_{\floor {Nt}+e_i+e_j}\right)] \,,\quad\,\{Nt_i\}\ge\{Nt_j\}\ge\{Nt_k\}
\end{align*}
where $t=(t_1\,,\,t_2\,,\,t_3)\in\overline V$ and $i,\,j,\,k\in \{1,\,2,\,3\}$ are pairwise different. 

Note that in both $d=2,3$ we have
\[
 \Psi_N(t)=\kappa N^{\frac{d-4}{2}}\varphi_{Nt}, \, \,\, t\in \frac{1}{N}\Z^d.
\]

From the above construction it follows that, for each $N$, $\Psi_N$ is a continuous function on $\overline V$.  This shows that $\Psi_N$ can be considered as a random variable taking values in $(C(\overline V), \mathcal C(\overline V))$ where $C(\overline V)$ is the space of 
continuous functions on $\overline V$ and $\mathcal C(\overline V)$ is its Borel $\sigma$-algebra. Also recall the definition of Green's function: 
the Green's function for the biharmonic operator is $G_V: V\times V\to \R$ such that for every fixed 
$x\in V$, it solves the equation
$$\Delta^2 G_V(x,y)= \delta_x(y), \qquad y\in V,$$
in the space $H^2_0(V)$, the completion of $C_c^\infty(V)$ with respect to the norm 
$$\|f\|_{H^2_0(V)}:=\|\nabla^2 f\|_{L^2(V)}.$$
In the above equations $\Delta^2$, the continuum bilaplacian, acts on the $y$ component, and $\nabla^2$ is the Hessian.
The detailed properties of such spaces are needed in $d\ge4$ so we defer the discussions on 
them to Section~\ref{sec:limit_des}. We denote the continuum Green's function by 
$G_V$ to indicate the dependence on the domain $V$.

We are now ready to state our main result for the case $d=2,\,3$. It shows that the convergence of the above described process occurs in the space of continuous functions.

\begin{theorem}[Scaling limit in $d=2,\,3$]\label{thm:low_d} 
Consider the interpolated membrane model $(\Psi_N(t))_{t\in \overline V}$ in $d=2$ and $3$ as above. 
Then there exists a centered continuous Gaussian process $\Psi$ with covariance $G_V(\cdot,\,\cdot)$ on 
$\overline V$ such that $\Psi_N$ converges in 
distribution to $\Psi$ in the space of all continuous functions on $\overline V$. 
Furthermore the process $\Psi$ is almost surely H\"older continuous with exponent $\eta$, 
for every $\eta\in (0,\,1)$ resp. $\eta\in(0,\,1/2)$ in $d=2$ resp. $d=3$. 
\end{theorem} 

An immediate consequence of the continuous mapping theorem is that, 
as $N\to\infty$, $$\sup_{x\in \overline V} \Psi_N(x)\overset{d}\to \sup_{x\in \overline V} \Psi(x).$$
It is easy to see that for any square or a cube $A$ in the $\frac1{N} \Z^d$ lattice,
$$\sup_{x\in A} \Psi_N(x)= \kappa N^{\frac{d-4}{2}} \max\{\vr_{Nx}: \, x\text{ is a vertex of A}\}.$$
Hence $\sup_{x\in \overline V} \Psi_N(x)= \kappa N^{\frac{d-4}{2}} \max_{x\in V_N} \vr_x$. 
So combining these observations we obtain the scaling limit of the maximum 
of the discrete membrane model in lower dimensions.
\begin{corollary}\label{cor:max}
Let $d\in\{2,\,3\}$ and let $M_N=\max_{x\in V_N} \vr_x$. Then as $N\uparrow \infty$ 
$$\kappa N^{\frac{d-4}{2}}M_N\overset{d}\to \sup_{x\in \overline V} \Psi(x).$$
\end{corollary}

\subsection{Proof of the scaling limit (Theorem~\ref{thm:low_d})}
The proof follows the general methodology of a functional CLT, namely, 
we first show the tightness of the interpolated field and secondly we 
show that the finite dimensional distributions converge. 
As a by-product of the proof, the limiting Gaussian process will be well-defined, that is, its covariance function will be positive definite. 
The finite dimensional convergence follows easily from the very recent work of 
\cite{Mueller:Sch:2017} where the convergence of the discrete Green's function to the continuum one is shown. Tightness also requires the crucial bounds on gradients 
which were derived in the same article. Since we have interpolated the field continuously and 
not piece-wise in boxes or cubes one of the main efforts is to deduce moment bounds from integer lattice points. 

\subsubsection{Tightness and H\"older continuity}
To derive the tightness we need the following ingredients. The first one consists in the following bounds 
for the discrete Green's function and its gradients which follow from 
\citet{Mueller:Sch:2017}. We define the directional derivative of a function $u:\Z^d\to\R$ as
\[
 D_i u(x):=u(x+e_i)-u(x),
\]
and the discrete gradient as
\[
 \nabla u(x)=(D_i u(x))_{i=1}^d.
\]
For functions of several variables we use a subscript
to indicate the variable with respect to which a derivative is taken, for example in $D_{i,\,1}D_{j,\,2}u(x,\,y)$ we take the discrete derivative in the direction $i$ in the variable $x$ and in $j$ in the variable $y$, and $\nabla_x G(x,\,y)$ means we are taking the gradient in the $x$ variable. We now state some bounds on the covariance function and its gradient from \cite{Mueller:Sch:2017}, where they appear in a more general version.
\begin{lemma}[{\citet[Theorem 1.1]{Mueller:Sch:2017}}] \label{DGF_bound0} Let $d\in\{2,\,3\}$. 
	\begin{enumerate}[ref=(\arabic*)]
		\item\label{1.DGF_bound0} For any $x,y\in\Z^d$
		$$\left|G_{N}(x,y)\right|\le C N^{4-d}.$$
		\item\label{2.DGF_bound0}  For any $x,y\in\Z^d$
		$$\|\nabla_x G_N(x,y)\|\le C N^{3-d}.$$
		
		\item\label{3.DGF_bound0} For any $x,y\in\Z^d$ 
		\begin{align*}\|\nabla_x\nabla_yG_{N}(x,y)\|\le \left\lbrace 
		\begin{array}{l l}
		C\log\left(1+\frac{N^2}{(\|x-y\|+1)^2}\right) & \quad \text{if $d=2$}\\
		C & \quad \text{if $d=3$}
		\end{array}. \right.
		\end{align*}
			\end{enumerate}
\end{lemma}
Now from the estimate \ref{3.DGF_bound0} and the fact that 
\[
 \E_N\left[\left( \vr_{z+e_i}-\vr_z\right)^2 \right]=D_{i,\,2}D_{i,\,1} G(z, z)
\]
one can observe the following Fact.
\begin{fact}\label{observation0}
	For $z\in\mathbb{Z}^d$ 
	$$\E_N\left[\left( \vr_{z+e_i}-\vr_z\right)^2 \right]\le \left\lbrace 
	\begin{array}{l l}
	C\log N & \quad \text{if $d=2$}\\
	C & \quad \text{if $d=3$}
	\end{array} .\right.$$
\end{fact}

Next we want to show that the sequence $\{\Psi_N\}_{N\in\N}$ is tight in $C(\overline V )$. 
We use the following theorem, whose proof follows from that of Theorem 14.9 of \cite{kallenberg:foundations}.
\begin{theorem}\label{thm:tightness0}
	Let $X^1, X^2, \ldots $ be continuous processes on $\overline V$ with values in a complete separable metric space $(S, \rho)$. Assume that $ (X_0^n)$ is tight in $S$ and that for constants $\alpha, \beta>0$
	\begin{equation}\label{eq:kallenberg:condition} 
	\E[ \rho(X_s^n, X_t^n)^\alpha ]\le C \|s-t\|^{d+\beta}, \quad s,\, t\in \overline V
	\end{equation}
	uniformly in $n$. Then $(X^n)$ is tight in $ C(\overline V, S)$ and for every $c\in (0, \beta/\alpha)$ 
	the limiting processes are almost surely H\"older continuous with exponent $c$.
\end{theorem}

Observe that the process $(\Psi_N(t))_{t\in \overline V}$ is Gaussian, and since from 
Lemma~\ref{DGF_bound0} it follows that $G_{N-1}(0,0) \le N^{4-d}$, it is easy to see that $(\Psi_N(0))$ is tight. Again, using the properties of 
Gaussian laws, to show~\eqref{eq:kallenberg:condition} it is enough to show the following the lemma. 
\begin{lemma}\label{moment_bound0} 
Let $b\in (0,\,1)$ in $d=2$ and $b=0$ in $d=3$. Then there exists a constant $C>0$ (which depends on $b$ in $d=2$) such that
	\begin{align}
	\E\left[ \left| \Psi_N(t)- \Psi_N(s)\right|^2\right]\le C \lVert t-s\rVert^{1+b} \label{eq:moment0}
	\end{align}
	for all  $t,s\in\overline V$, uniformly in $N$.
\end{lemma}
This Lemma will immediately give \eqref{eq:kallenberg:condition} and hence the H\"older continuity of the limiting field.
\begin{corollary}\label{cor:holder}
The field $\Psi$ is almost surely H\"older continuous with exponent $\eta$, where $\eta\in (0,\,1)$ in $d=2$ and $\eta\in(0,\,1/2)$ in $d=3$.
\end{corollary}
\begin{proof}
We note that for $t,\,s\in\overline V$, the random variable $\Psi_N(t)-\Psi_N(s)$ is Gaussian. Therefore 
using Lemma \ref{moment_bound0} we have, for any $\alpha$ such that ${(1+b)\alpha}/2 >d$, that there 
is a constant $C$ such that the following holds uniformly in $N$ with $\beta:={(1+b)\alpha}/2-d:$ 
$$\E[ \left|\Psi_N(t)-\Psi_N(s)\right|^{\alpha} ]\le C \|t-s\|^{d+\beta}, \quad s,\, t \in \overline V.$$
The conclusion follows then from Theorem~\ref{thm:tightness0}.
\end{proof}
Now we show the proof of the Lemma.
\begin{proof}[Proof of Lemma~\ref{moment_bound0}]First we consider $d=2$. We fix a $b\in(0,\,1)$ and 
let $t,\,s\in\overline V$. We split the proof into a few cases.
\begin{enumerate}[label=\textbf{Case \arabic*}:,wide=\parindent]
\item Suppose $t, s$ belong to the same smallest square box in the lattice $\frac1N\Z^2$. 	
First assume $\lfloor Nt\rfloor = \lfloor Ns\rfloor$, that is, 
the points are in the interior and not touching the top and right boundaries. In this case if we have $\{Nt_1\}\ge\{Nt_2\}$ and $\{Ns_1\}\ge\{Ns_2\}$. Then by definition of the interpolation we have
	\begin{align*}
	\Psi_N(t)-\Psi_N(s)&=\kappa[\left(t_1-s_1\right)\left(\vr_{\lfloor Nt\rfloor+e_1}-\vr_{\lfloor Nt\rfloor}\right) \\
	&+ \left( t_2-s_2\right)\left(\vr_{\lfloor Nt\rfloor+e_1+e_2}-\vr_{\lfloor Nt\rfloor+e_1}\right)].
	\end{align*}
So from the above expression we have	
	\begin{align*}
	&\E\left[\left( \Psi_N(t)- \Psi_N(s)\right)^2\right]\le 2\kappa^2[(t_1-s_1)^2 \E[ \left( \vr_{\floor{Nt}+e_1}- \vr_{\floor{Nt}}\right)^2]\\
	&\qquad + (t_2-s_2)^2\E[\left( \vr_{\floor{Nt}+e_1+e_2}-\vr_{\floor{Nt}+e_1}\right)^2]].
	\end{align*}
	Now from Fact~\ref{observation0} and $\left|t_1-s_1\right|, \left|t_2-s_2\right|<N^{-1}$ we obtain \eqref{eq:moment0}. The argument is similar if one has $\{Nt_1\}\le\{Nt_2\}$ and $\{Ns_1\}\le\{Ns_2\}$.\\
	
	Again if $\{Nt_1\}\ge\{Nt_2\}$ and $\{Ns_1\}<\{Ns_2\}$, or if $\{Nt_1\}<\{Nt_2\}$ and $\{Ns_1\}\ge\{Ns_2\}$ then we consider the point $u$ on the line segment joining $t$ and $s$ such that $Nu$ 
	is the point of intersection of the line segment joining $Nt, Ns$ and the 
	diagonal joining $\floor{Nt}, \floor{Nt}+e_1+e_2.$ Then we have using the above 
	computations
	\begin{align*}
	\E\left[ \left| \Psi_N(t)- \Psi_N(s)\right|^2\right]&\le2\E\left[ \left| \Psi_N(t)- \Psi_N(u)\right|^2\right] + 2\E\left[ \left| \Psi_N(u)- \Psi_N(s)\right|^2\right]\\
	&\le C\left[\lVert t-u\rVert^{1+b}+\lVert u-s\rVert^{1+b}\right]\le C\lVert t-s\rVert^{1+b}.
	\end{align*} 
Now the other case, that is, when $\lfloor Nt\rfloor \neq \lfloor Ns\rfloor$ follows from above by continuity.

\item 
Suppose $t, s$ do not belong to the same smallest square box in the lattice $\frac1N\Z^2$. In this case if $\|t-s\| \le 1/N$ then one can obtain \eqref{eq:moment0} by the above case and a suitable point in between. So we assume $\|t-s\|>1/N$. Depending on whether $Nt$ and $Ns$ belong to the discrete lattice we split the proof in two broad cases. We will use bounds on mixed discrete 
derivatives for a better control of finite differences of the Green's function.
\begin{enumerate}[label=\underline{Sub-case 2 (\alph*)},ref=Sub-case 2 (\alph*),wide=\parindent]
\item\label{2(a)} Suppose $t,s\in\frac 1N\Z^2$. Then
	\begin{align*}
	\E\left[ \left| \Psi_N(t)- \Psi_N(s)\right|^2\right]&=\frac{\kappa^2}{N^2}\left[ G_{N-1}(Nt,Nt)-G_{N-1}(Ns,Nt)\right.\\
	&\left.-G_{N-1}(Nt,Ns)+G_{N-1}(Ns,Ns)\right].
	\end{align*}
We assume without loss of generality $Ns_1\le Nt_1,\, Ns_2\le Nt_2$. Also denote $M:= N(t_1-s_1+t_2-s_2)$ and let $(u_i)_{i=0}^M$ be 
such that $u_i=s+i/N e_1$ for $i\le N(t_1-s_1)$ and $u_i=s+(t_1-s_1)e_1+(i/N-(t_1-s_1))e_2$ for $i>N(t_1-s_1)$. Then
	\begin{align*}
	\E&\left[ \left| \Psi_N(t)- \Psi_N(s)\right|^2\right]
	=\frac{\kappa^2}{N^2}\sum_{i=0}^{M-1}\left[G_{N-1}(Nu_{i+1},Nt)-G_{N-1}(Nu_i,Nt) \right]\\
	& -\left[ G_{N-1}(Nu_{i+1},Ns)-G_{N-1}(Nu_i,Ns) \right]\\
	& =\frac{\kappa^2}{N^2}\sum_{i,\,j=0}^{M-1}\left[ G_{N-1}(Nu_{i+1},Nu_{j+1})-G_{N-1}(Nu_{i+1},Nu_j)\right.\\
	& \left.-G_{N-1}(Nu_i,Nu_{j+1})+G_{N-1}(Nu_i,Nu_j)\right]{\le} \frac{C}{N^2}\sum_{i,\,j=0}^{M-1}\log\left(1+\frac{N^2}{(\|Nu_i-Nu_j\|+1)^2}\right)\\
	\end{align*}
	where we have used Lemma~\ref{DGF_bound0}~\ref{3.DGF_bound0} in the last inequality and we have 
	absorbed the constant $\kappa^2$ in the generic constant $C$. Now using the definition of $u_i,\,u_j$ the right-hand side above 
	is bounded above by
	\begin{align*}
	\frac{C}{N^2}&\sum_{i,\,j=0}^{M-1}\log\left(1+\frac{N^2}{\Big(\frac{\left|i-j\right|}{\sqrt 2}+1\Big)^2}\right)\le \frac{C}{N^2}\sum_{i,\,j=0}^{M-1}\log\left(1+\frac{N}{(\left|i-j\right|+1)}\right)\\
	&\le \frac{CM}{N^2}\sum_{l=-M+1}^{M-1}\log\left(1+\frac{N}{(\left|l\right|+1)}\right)\le 
	\frac{CM}{N}\int_{0}^{\frac{M}{N}}\log\left(1+\frac 1{x}\right)\mathrm{d}x\\
	&\le C \Big(\frac{M}{N}\Big)^2\left[1+\log\Big(1+\frac{N}{M}\Big)\right]\le C \|t-s\|^{1+b}.
	\end{align*}
\item\label{2(b)}Suppose at least one between $t,s$ does not belong to $\frac 1N\Z^2$. Then
	\begin{align*}
	\E\left[ \left| \Psi_N(t)-\right.\right.&\left. \left.\!\Psi_N(s)\right|^2\right]\le 3\E\left[ \left| \Psi_N\left(t\right)- \Psi_N\left(\frac{\floor{Nt}}{N}\right)\right|^2\right]\\
	&+3\E\left[ \left| \Psi_N\left(\frac{\floor{Nt}}{N}\right)- \Psi_N\left(\frac{\floor{Ns}}{N}\right)\right|^2\right]+3\E\left[ \left| \Psi_N\left(\frac{\floor{Ns}}{N}\right)- \Psi_N\left(s\right)\right|^2\right]\\
	&\le C\left[\left\|t-\frac{\floor{Nt}}{N}\right\|^{1+b}+\left\|\frac{\floor{Nt}}{N}-\frac{\floor{Ns}}{N}\right\|^{1+b}+
	\left\|\frac{\floor{Ns}}{N}-s\right\|^{1+b}\right]\le C\|t-s\|^{1+b}.
	\end{align*}	
Note that for the last inequality we have used our assumption $\|t-s\|>1/N$.	
\end{enumerate}

\end{enumerate}
Now we consider $d=3$. Let $t,\,s\in\overline V$. We split the proof into cases similar to those of $d=2$. We give a brief description.
For Case 1, suppose $t, s$ belong to the same smallest cube in the lattice $\frac1N\Z^3$. First assume $\lfloor Nt\rfloor = \lfloor Ns\rfloor$. In this case if $\{Nt_1\}\ge\{Nt_2\}\ge\{Nt_3\}$ and $\{Ns_1\}\ge\{Ns_2\}\ge\{Ns_3\}$ then it follows from the definition of interpolation
	\begin{align*}
	\E\left[\left( \Psi_N(t)- \Psi_N(s)\right)^2\right]&\le 3N\kappa^2[(t_1-s_1)^2 \E[ \left( \vr_{\floor{Nt}+e_1}- \vr_{\floor{Nt}}\right)^2]\\
	&+ (t_2-s_2)^2\E[\left( \vr_{\floor{Nt}+e_1+e_2}-\vr_{\floor{Nt}+e_1}\right)^2]\\
	&+(t_3-s_3)^2\E[\left( \vr_{\floor{Nt}+e_1+e_2+e_3}-\vr_{\floor{Nt}+e_1+e_2}\right)^2]].
	\end{align*}
Now from Fact~\ref{observation0} and the fact that $\left|t_1-s_1\right|, \left|t_2-s_2\right|, \left|t_3-s_3\right|<1/N$ we have ~\eqref{eq:moment0}.
Note that this is a particular case of $t,s$ lying in the same tetrahedral portion of the cube. Hence if $t,\,s$ lie in the 
same tetrahedral portion of the cube then by similar arguments~\eqref{eq:moment0} holds.
If $t,s$ do not lie in the same tetrahedral part then we consider points 
(at most 3) on the line segment joining them such that two consecutive between $t$, 
the selected points and $s$ lie in the same tetrahedral part. 
Then applying the previous argument we can obtain~\eqref{eq:moment0}.
Now the case when $\lfloor Nt\rfloor \neq \lfloor Ns\rfloor$ follows by continuity.
 For Case 2, we describe Sub-case 2(a) which turns out to be simpler in $d=3$. The rest of the argument is similar to that in $d=2$.
Suppose $t,s\in\frac 1N\Z^3$ with $\|t-s\|>1/N$.  Then
	\begin{align*}
	\E\left[ \left| \Psi_N(t)- \Psi_N(s)\right|^2\right]=\frac{\kappa^2}{N}\left[ G_{N-1}(Nt,Nt)-G_{N-1}(Ns,Nt)-G_{N-1}(Nt,Ns)+G_{N-1}(Ns,Ns)\right]
	\end{align*}
		Without loss of generality assume $Ns_1\le Nt_1,\, Ns_2\le Nt_2,\,Ns_3\le Nt_3$. Then
	\begin{align*}
	&G_{N-1}(Nt,Nt)-G_{N-1}(Ns,Nt)=\sum_{i=1}^{N(t_1-s_1)} D_{1,\,1} G_{N-1}(Ns+(i-1)e_1,Nt) \\
	&+\sum_{j=1}^{N(t_2-s_2)} D_{2,\,1} G_{N-1}(Ns+N(t_1-s_1)e_1+(j-1)e_2,Nt)\\
	&+\sum_{l=1}^{N(t_3-s_3)} D_{3,\,1} G_{N-1}(Ns+N(t_1-s_1)e_1+N(t_2-s_2)e_2+(l-1)e_3,Nt)\\
	&\overset{\ref{2.DGF_bound0}}{\le} C\left(N(t_1-s_1)+ N(t_2-s_2)+N(t_3-s_3)\right)\le CN\|t-s\|.
	\end{align*}
	Hence \eqref{eq:moment0} follows.
 \end{proof}

\subsubsection{Finite dimensional convergence}
The main content of this Subsubsection is to show
\begin{proposition}With the notation of Theorem~\ref{thm:low_d}, for all $s,\,t\in\overline V$,
 \[
  \lim_{N\to\infty}\mathrm{Cov}(\Psi_N(t),\Psi_N(s))= \mathrm{Cov}(\Psi(t),\Psi(s)).
 \]
\end{proposition}
\begin{proof}
To show the finite dimensional convergence we use Corollary 1.4 of \cite{Mueller:Sch:2017} (in their setting the domain was $(0,\,1)^d$ but the result works for $V$ as well). 
We observe that for $h:=1/N$,  one has $G_{N-1}(x,y)= 4d^2h^{d-4}G_h(hx,hy)$ where $G_h$ 
satisfies for $x\in\mathrm{int}(V_h)$ with $V_h=[-1,1]^d\cap h\Z^d$ the following boundary value problem ($\Delta_h$ is defined in Appendix~\ref{appendix:T}):
\[\begin{cases}
\Delta_h^2 G_{h}(x,y) = \frac1{h^d}\delta_{x}(y) & y\in \mathrm{int}(V_h) \\
G_{h}(x,y) = 0 & y\notin \mathrm{int}(V_h)
\end{cases}.
\]
Let $\Psi$ be the Gaussian process on $\overline V$ such that $\E[\Psi(t)\Psi(s)]= G_V(t,s)$ for all $t,\,s\in\overline V$, 
where $G_V$ is the Green's function for the biharmonic equation with homogeneous Dirichlet 
boundary conditions (it will be a by-product of this proof that such a process exists).
First we consider $d=2$. For $t\in\overline V$ we have 
\begin{align*}
\Psi_N(t)
=\Psi_{N,1}(t)+\Psi_{N,2}(t) ,
\end{align*}

where $\Psi_{N,1}(t)=\frac \kappa{N}\vr_{\floor{Nt}}$ and
\begin{align*}
\Psi_{N,2}(t)&=\frac \kappa N\sum_{i,j\in\{1,2\}, i\neq j}\one_{\left(\{Nt_i\}\ge\{Nt_j\}\right)}(t)[\{Nt_i\}\left(\vr_{\floor{Nt}+e_i}-
\vr_{\floor{Nt}}\right)\\
&+\{Nt_j\}\left(\vr_{\floor{Nt}+e_i+e_j}-\vr_{\floor{Nt}+e_i}\right)].
\end{align*}

Then using Fact~\ref{observation0} we have 
$\var(\Psi_{N,2}(t))\le C(\log N)N^{-2}$ and hence $\Psi_{N,2}(t)$ converges to zero in probability as $N$ tends to infinity.

 Again if $t\in V$ then $$\var(\Psi_{N,1}(t))= \frac {\kappa^2}{N^2}G_{N-1}(\floor{Nt},\floor{Nt})=G_h(h\floor{Nt},h\floor{Nt})$$
 and $G_h(h\floor{Nt},h\floor{Nt})$ converges to $G_V(t,t)$ by Corollary 1.4 of \cite{Mueller:Sch:2017}. Also if $t\in\partial V$ then $\var(\Psi_{N,1}(t))=0=G_V(t,t)$. Hence $\Psi_N(t)\overset{d}\to \Psi(t)$. 
 
Similarly one can show using Lemma~\ref{DGF_bound0}, Fact~\ref{observation0} and \citet[Corollary 1.4]{Mueller:Sch:2017} that 
for any $t,s\in\overline V$, $$\mathrm{Cov}(\Psi_N(t),\Psi_N(s))\to \mathrm{Cov}(\Psi(t),\Psi(s)).$$ Since these variables under consideration are Gaussian, the finite dimensional follows from the convergence of the covariance. 

In $d=3$, for $t\in\overline V$ we have 

\begin{align*}
\Psi_N(t) &=\frac \kappa{\sqrt N}\vr_{\floor{Nt}}+\frac \kappa{\sqrt N}\sum_{i,\,j,\,k\in\{1,\,2,\,3\},\,\text{pairwise different}}\one_{\left(\{Nt_i\}\ge\{Nt_j\}\ge\{Nt_k\}\right)}(t)\\
&[\{Nt_i\}\left(\vr_{\floor{Nt}+e_i}-\vr_{\floor{Nt}}\right)+\{Nt_j\}\left(\vr_{\floor{Nt}+e_i+e_j}-\vr_{\floor{Nt}+e_i}\right) \\
&+\{Nt_k\}\left(\vr_{\floor{Nt}+e_i+e_j+e_k}-\vr_{\floor{Nt}+e_i+e_j}\right)]\\
&=: \Psi_{N,1}(t)+\Psi_{N,2}(t).
\end{align*}
By means of Fact~\ref{observation0} we have $\var(\Psi_{N,2}(t))\le C/{N}$ and hence $\Psi_{N,2}(t)$ converges to zero in probability as $N\to\infty$. The rest of the 
proof is the same as $d=2$ and follows from Corollary 1.4 of \cite{Mueller:Sch:2017}.
\end{proof}

\section{Convergence of finite volume measure in \texorpdfstring{$d\ge 4$}{}}\label{sec:limit_des}

In this Section $D$ denotes a bounded domain in $\R^d$, $d\ge 4$, with smooth boundary. 
\begin{remark}[Regularity of the boundary of the domain]In what follows, the assumption of smoothness of the boundary is required to obtain asymptotics of the eigenvalues of the biharmonic operator (cf. Proposition~\ref{prop:Weyl}). 
\end{remark}

\subsection{Description of the limiting field}
\subsubsection{Spectral theory for the biharmonic operator}
Let $C_c^\infty(D)$ denote the space of 
infinitely differentiable functions $u: D\to \R$ with compact support inside 
$D$. For $\alpha= (\alpha_1, \,\ldots,\, \alpha_d)$ a multi-index define
 $$D^\alpha u= \frac{\partial^{\alpha_1}}{\partial x_1^{\alpha_1}}\cdots \frac{\partial^{\alpha_d}}{\partial x_d^{\alpha_d}} u.$$
Suppose $f, \,g\in L^1_{loc}(D)$. One says that $g$ is the $\alpha$-th weak partial derivative of 
$f$ (written $D^\alpha f= g$) if 
 $$\int_{D} f D^\alpha u \De x= (-1)^{|\alpha|} \int_{D} g u \De x \quad\forall\, u \in C_c^\infty(D).$$
The Sobolev space $W^{k,p}$ is defined in the usual way as
$$W^{k,p}= \{ f\in L^1_{loc}(D) : \,D^\alpha f\in L^p(D), \, |\alpha|\le k\}.$$
Denote by $H^k(D):= W^{k,2}(D)$, $k=0,\,1,\,\ldots$, which is a Hilbert space with norm
 $$\|f\|_{H^k(D)} = \left( \sum_{|\alpha|\le k}\int_{D} |D^\alpha f|^2\De x\right)^{1/2}.$$
It is true that if $a>b$ then $H^a(D)\subset H^b(D)$. Let us define another Hilbert space,
$$H^k_0(D):= \overline{ C_c^{\infty}(D)}^{\|\cdot\|_{H^k(D)}}$$
and let $H^{-k}(D)= [ H^k_0(D)]^*$ be its dual. In this Section 
we will use round brackets $(\cdot,\,\cdot)$ to denote the action of a dual Hilbert space on the original space, and $\la\cdot,\,\cdot\ra$ for inner products.
We consider the inner product
$$\left\langle u, v\right\rangle _{H^2_0}= \int_D \Delta u \Delta v \De x$$
 which induces a norm on $H^2_0(D)$ equivalent to the standard Sobolev norm \cite[Corollary~2.29]{GGS}. We always consider $H^2_0(D)$ with this norm. 

We review briefly the spectral theory for the biharmonic operator as it helps us 
to give an explicit construction of the continuum bilaplacian field. We have the following Theorem, which basically says that we can construct an operator $B$ being the inverse of the bilaplacian (see also Remark~\ref{rem:equiv_index}). 
 \begin{theorem}\label{thm:B_0}
There exists a bounded linear isometry
$$B_0:H^{-2}(D)\rightarrow H^2_0(D)$$ 
such that, for all $f\in H^{-2}(D)$ and for all $v\in H^2_0(D)$, $$(f,\,v)= \la v,\,B_0 f\ra_{H^2_0}.$$ Moreover, the restriction $B$ on $L^2(D)$ of the operator $i\circ B_0 :H^{-2}(D)\rightarrow L^2(D) $ 
is a compact and self-adjoint operator, where $i: H^2_0(D)\hookrightarrow L^2(D)$ is the inclusion map.
 \end{theorem}
 \begin{proof}
 	Fix $f\in H^{-2}(D)$. By the Riesz representation theorem there exists a unique $u_f\in H^2_0(D)$ such that for all $v\in H^2_0(D)$ 
 	$$(f,\,v)= \la v, u_f\ra_{H^2_0}.$$ 
 	We define $B_0 f:=u_f$. Then by definition $B_0$ is a bounded linear isometry and for all $v\in H^2_0(D)$ 
 	$$(f,\,v)= \la v, B_0f\ra_{H^2_0}.$$
 	We have $H^2_0(D)\hookrightarrow H^1_0(D)\hookrightarrow L^2(D)$ and the second embedding is compact. 
 	So $i: H^2_0(D)\hookrightarrow L^2(D)$ is compact and hence the operator $i\circ B_0 :H^{-2}(D)\rightarrow L^2(D) $ is compact. This implies that the restriction $B$ is compact. $B$ is self-adjoint as for any $f, g\in L^2(D)$, 
 	\[\la Bf, g\ra_{L^2}= (g,\, Bf)=\la Bf, Bg\ra_{H^2_0}= (f,\, Bg)= \la f, Bg\ra_{L^2}.\qedhere\] 	
 \end{proof}
Consequently we can find now an orthonormal basis of elements of $H_0^2(D)$, as the next theorem shows.
 \begin{theorem}\label{thm:eigenfunctions}
 	There exist $u_1,\, u_2,\, \ldots $ in $H^2_0(D)$ and numbers $$0<\lambda_1\le\lambda_2\le \cdots \to\infty$$ such that
	\begin{itemize}
		\item $\{u_j\}_{j\in \N}$ is an orthonormal basis for $L^2(D)$,
		\item $Bu_j=\lambda^{-1}_ju_j$, where $B$ is as in Theorem~\ref{thm:B_0},
 		\item $(u_j,v)_{H^2_0}=\lambda_j\la u_j, v\ra_{L^2}$  for all $v\in H^2_0(D)$,
 		\item $\{\lambda^{-1/2}_ju_j\}$ is an orthonormal basis for $H^2_0(D)$.
 	\end{itemize}
 \end{theorem}
 \begin{proof}
 	 By the spectral theorem for compact self-adjoint operators we get an orthonormal basis of $L^2(D)$ 
 	 consisting of eigenvectors of $B$ 
 	 with $Bu_j= \widetilde \lambda_j u_j$ and eigenvalues $\widetilde \lambda_j \to 0$. Note that for any $f\in L^2(D)$, $Bf=0$ implies that
 	 $$\la v, \,f\ra_{L^2}= \la v, \,Bf\ra_{H^2_0}=0 \quad \forall v\in H^2_0(D)$$
 	 and hence $\la g, f\ra_{L^2}=0$ for all $g\in L^2(D)$ (since $H^2_0(D)$ is dense in $L^2(D)$) and so $f\equiv 0$. Thus $0$ is not an eigenvalue of $B$ and we have for any $j\in \mathbb{N}$
 	 $$u_j = \frac1{\widetilde \lambda_j} B u_j = B\frac{ u_j}{\widetilde \lambda_j} \in \mathrm{Range}(B)\subset H^2_0(D).$$
	 Hence $u_j\in H^2_0(D)$. Now observe that, for any $j\in \mathbb{N}$,  
 	 $\widetilde \lambda_j \la u_j,\, v\ra_{H^2_0}= \la Bu_j, \,v\ra_{H^2_0} = \la u_j,\, v\ra_{L^2}$ for all $v\in H^2_0(D)$. So this gives
 	 $$\widetilde \lambda_j \la u_j,\, u_j\ra_{H^2_0}= \|u_j\|_{L^2}=1.$$
 	 But $\la u_j, u_j\ra_{H^2_0}>0$ and hence $\tilde \lambda_j>0$ for all $j\in \mathbb{N}$. We define $\lambda_j:=1/{\widetilde \lambda_j}$. So we can conclude $$0<\lambda_1\le\lambda_2\le \ldots \to\infty.$$
 	 Moreover $Bu_j=\lambda^{-1}_j u_j$ and
 	 \begin{equation}\label{eq:eigen_u}
 	 \la u_j, \,v\ra_{H^2_0}= \lambda_j \la u_j, v\ra_{L^2} \quad \forall \, v\in H^2_0(D).
 	 \end{equation}
 	 We now show that $\{\lambda^{-1/2}_j u_j\}_{j\in\N}$ is an orthonormal basis for $H^2_0(D)$. Indeed we have 
 	 \begin{align*}
	 \la\lambda^{-\frac12}_ju_j, \lambda^{-\frac12}_ku_k \ra_{H^2_0} &=\lambda^{-\frac12}_j\lambda^{-\frac12}_k\la u_j, u_k \ra_{H^2_0}\\
 	 &=\lambda^{\frac12}_j\lambda^{-\frac12}_k\la u_j, u_k \ra_{L^2}=\delta_{jk}.
 	 \end{align*}
 	 So $\{\lambda^{-1/2}_ju_j\}$ is an orthonormal system. But for any $v\in H^2_0(D)$, $\la u_j, v\ra_{H^2_0}=0 $ for all $j$ implies that $\la u_j, v\ra_{L_2}= 0$ for all $j$ which in turn implies $v=0$. This completes the proof.            
  \end{proof}
  \begin{corollary}\label{cor:eigen_plus_smooth}
  For each $j\in \N$ one has $u_j\in C^\infty(D).$ Moreover $u_j$ is an eigenfunction of $\Delta^2$ with eigenvalue $\lambda_j$.
  \end{corollary}
  \begin{proof}
  We have for all $v\in H_0^2(D)$:
  \[   \la \Delta^2 u_j,\,v \ra_{L^2}\stackrel{GI}{=}\la u_j,\,v \ra_{H_0^2}\stackrel{\text{Theorem~\ref{thm:eigenfunctions}}}{=}\lambda_j\la u_j,\,v \ra_{L^2}\]
  where ``GI'' stands for Green's first identity
  \[
  \int_D u\Delta v \De V=-\int_D\nabla u\cdot\nabla v \De V+\int_{\partial D}u\nabla v\cdot\mathbf n \De S.
  \]
  Thus $u_j$ is an eigenfunction of $\Delta^2$ with eigenvalue $\lambda_j$ in the weak sense. The smoothness of 
  $u_j$ follows from the fact that $\Delta^2$ is an elliptic operator with smooth coefficients 
  and the elliptic regularity theorem \cite[Theorem~9.26]{FollandReal}. Hence $u_j$ is an 
  eigenfunction of $\Delta^2$ with eigenvalue $\lambda_j$.
 \end{proof}

\begin{remark}\label{rem:series} As a consequence of the above, one easily has that
$$\|f\|_{H^2_0}^2 = \sum_{j\ge 1} \lambda_j \la f, u_j\ra^2_{L^2}$$
for any $f\in H^2_0(D)$.
\end{remark}

We conclude this subsection with some bounds for the derivatives of the eigenfunctions $u_j$ of Theorem~\ref{thm:eigenfunctions}.
\begin{lemma}\label{lem:bounds_GGS}
	The following bounds hold:
	\begin{align}
	\sup_{x\in \overline{D}}\lvert u_j(x)\rvert\le C\lambda_j^{l_0} \label{bound0}, \\
	\sum_{\lvert\alpha\rvert\le2}\sup_{x\in \overline{D}}\lvert D^\alpha u_j(x)\rvert\le C\lambda_j^{l_2} \label{bound1}, \\
	\sum_{\lvert\alpha\rvert\le5}\sup_{x\in \overline{D}}\lvert D^\alpha u_j(x)\rvert\le C\lambda_j^{l_5} \label{bound2}
	\end{align}
	where 
	$$l_m:=\ceil*{ \frac{1}{4}\left(\floor*{\frac d2}+m+1\right)},\quad m=0, 2, 5.$$
\end{lemma}
\begin{proof}
	Taking $l_0=\lceil {1}/{4}(\floor{d/2}+1)\rceil$ we obtain from \citet[Chapter~5, Theorem~6 (ii)]{Evans} that $\sup_{x\in \overline{D}}\lvert u_j(x)\rvert\le C \| u_j \|_{H^{4l_0}(D)}$. Now a repeated application of \citet[Corollary~2.21]{GGS} gives
	\begin{align*}
	\sup_{x\in \overline{D}}\lvert u_j(x)\rvert\le C \| u_j \|_{H^{4l_0}(D)} \le C \lambda_j \| u_j \|_{H^{4l_0-4}(D)} \le \cdots \le C \lambda_j^{l_0}.
	\end{align*}
	The other two bounds are obtained similarly. We make a passing remark that the smoothness of the boundary is needed in the results quoted above. 
\end{proof}
 \subsubsection{Definition of the limiting field via Wiener series}\label{subsub:const}
 For any $v\in C_c^\infty(D)$ and for any $s>0$ we define $$\|v\|_s^2:=\sum_{j\in \N}\lambda_j^{s/2}\la v,u_j\ra_{L^2}^2.$$ We define $\mathcal H_0^s(D)$ to be the Hilbert space completion of $C_c^\infty(D)$ 
 with respect to the norm $\|\cdot\|_s$.
 Then $\left(\mathcal H_0^s(D) \,,\,\|\cdot\|_s\right) $ is a Hilbert space for all $s>0$.
 \begin{remark}\label{rem:dual_spaces}\leavevmode
 	\begin{itemize}
 		\item Note that for $s=2$ we have $\mathcal H_0^2(D)= H_0^2(D)$ by Remark~\ref{rem:series}.
 		\item $i:\mathcal H_0^s(D)\hookrightarrow L^2(D)$ is a continuous embedding.
 	\end{itemize}
 \end{remark}
 \paragraph{Dual spaces.}
 For $s>0$ we define $\mathcal H^{-s}(D)= (\mathcal H_0^s(D))^*$, the dual space of $\mathcal H_0^s(D)$. Then we have $$\mathcal H_0^s(D) \subseteq L^2(D) \subseteq  \mathcal H^{-s}(D).$$
 One can show using the Riesz representation theorem that for $s>0$ the norm of $\mathcal H^{-s}(D)$ is given by
 \[
  \|v\|_{-s}^2:=\sum_{j\in \N}\lambda_j^{-s/2}(v,\,u_j)^2,\qquad v\in \mathcal H^{-s}(D).
 \]
Recall that $\left(\cdot,\,\cdot\right)$ denotes the action of the dual space $\mathcal H^{-s}(D)$ on $\mathcal H_0^s(D)$. Moreover, for $v\in L^2(D)$ we have
\[
  \|v\|_{-s}^2:=\sum_{j\in \N}\lambda_j^{-s/2}\la v,\,u_j\ra_{L^2}^2. 
 \]

 	Before we show the definition of the continuum membrane model, we need an analog of Weyl's law for the eigenvalues of the biharmonic operator.
 \begin{proposition}[{\citet[Theorem~5.1]{Beals:1967}, \citet{Pleijel:1950}}]\label{prop:Weyl}
 		There exists an explicit constant $c$ such that, as $j\uparrow+\infty$,
 		\[
 		\lambda_j\sim c^{-d/4}j^{4/d}.
 		\]
 \end{proposition}
 	The result we will prove now shows the well-posedness of the series expansion for $\psi_D$.
 	\begin{proposition}\label{prop:series_rep_h}
 		Let $(\xi_j)_{j\in \N}$ be a collection of i.i.d. standard Gaussian random variables. Set
 		\[
 		\psi_D:=\sum_{j\in \N}\lambda_j^{-1/2}\xi_j u_j.
 		\]
 		Then $\psi_D\in \mathcal H^{-s}(D)$ a.s. for all $s>({d-4})/2$.
 	\end{proposition}
 	\begin{proof}
 		Fix $s>({d-4})/2$. Clearly $u_j\in L^2(D)\subseteq \mathcal H^{-s}(D)$. We need to show that $\|\psi_D\|_{-s}<+\infty$ almost surely. Now this boils down to showing the finiteness of the random series
 		$$\|\psi_D\|_{-s}^2=\sum_{j\ge 1} \lambda_j^{-s/2} \left(\sum_{k\ge 1} \lambda_k^{-1/2} u_k \xi_k\, , u_j\right)^2=\sum_{j\ge 1} \lambda_j^{-\frac{s}{2}-1}\xi_j^2 $$
 		where the last equality is true since $(u_j)_{j\ge 1}$ form an orthonormal basis of $L^2(D)$. Observe that the assumptions of Kolmogorov's two-series theorem are satisfied: indeed using Proposition~\ref{prop:Weyl} one has 
 		\[\sum_{j\ge 1}\E\left(\lambda_j^{-\frac{s}{2}-1}\xi_j^2\right)\asymp c\sum_{j\ge 1} j^{-\frac{4}{d}\left(\frac{s}{2}+1\right)}<+\infty\]
 		for $s>(d-4)/2$ and 
 		\[\sum_{j\ge 1}\var\left(\lambda_j^{-\frac{s}{2}-1}\xi_j^2\right)\asymp c\sum_{j\ge 1} j^{-\frac{4}{d}(s+2)}<+\infty\] 
 		for $s>(d-8)/4$. The result then follows.
 	\end{proof}

\subsection{Definition of the limiting field via abstract Wiener spaces}\label{par:AWS}
We want now to connect the series representation given in Proposition~\ref{prop:series_rep_h} with an equivalent characterisation of $\psi_D$. This alternative definition can be given through the theory of 
abstract Wiener space (AWS). For a comprehensive overview of the theory we refer the readers to \cite{Stroock:2010} 
for example. For our purposes it will suffice to recall that an abstract Wiener space is a triple $\left(\Theta,\,H,\,\mathcal{W}\right)$, where
\begin{itemize}
\item $\Theta$ is a separable Banach space,
\item $H$ is a Hilbert space which is continuously embedded as a dense subspace of $\Theta$, equipped with the scalar product $\la\cdot,\,\cdot\ra_H$,
\item $\mathcal W$ is a  Gaussian probability measure on $\Theta$ defined as follows.
\end{itemize}
Let $\Theta^*$ be the dual space of $\Theta$. Given any $x^*\in \Theta^*$ there exists a unique $h_{x^*}\in H$ such that for all $h\in H$,
$ (h,x^*)=\la h,h_{x^*}\ra_H$ where $(\cdot,\,x^*)$ denotes the action of $x^*$ on $\Theta$. The $\sigma$-algebra $\mathcal B(\Theta)$  on $\Theta$ is
such that all the maps $\theta\mapsto (\theta, \,x^* )$ are measurable.  $\mathcal W$ is a probability measure such that, for all $x^*\in \Theta^*$,

\begin{equation}\label{eq:aws}
\mathsf E_{\mathcal W}\left[\exp\left(\iota(\cdot, x^*)\right)\right]=\exp\left(-\f{\|{h_{x^*}}\|^2_H}{2}\right)  .
\end{equation}
In other words, the variable $(\cdot, x^*)$ under $\mathcal W$ is a centered Gaussian with variance $\|{h_{x^*}}\|^2_H$.
Next, 
we introduce the \emph{Paley--Wiener
map} $\mathcal{I}$. $\mathcal I$ is viewed as a mapping
\begin{eqnarray*}
\mathcal I:\,h_{x^*}\in H&\mapsto& \mathcal I(h_{x^*})\in L^2(\mathcal W)\\
 && \theta\in \Theta \mapsto [\mathcal I(h_{x^*})](\theta):=(\theta,\,x^*).
\end{eqnarray*}
Since $\{h_{x^*}:\,x^*\in\Theta^*\}$ is dense in $H$, the map $h_{x^*}\mapsto I(h_{x^*})$ can be uniquely extended as a linear isometry from $H$ to $L^2(\mathcal W)$.
\citet[Theorem~8.2.6]{Stroock:2010} yields that the family of Paley--Wiener integrals $\left\{ \mathcal{I}\left(h\right):\,h\in H\right\} $ is Gaussian, where each $\mathcal I(h)$ has 
mean zero and variance $\|h\|_H^2$. Given~\eqref{eq:aws} the family $\{\mathcal I(u_j):\,\{u_j\}_{j\in\N}\text{ orthonormal basis of } H\}$ 
is formed by i.i.d. standard Gaussians. 

In our setting, by combining 
\citet[\S8.3.2]{Stroock:2010} and the Wiener series given in Proposition~\ref{prop:series_rep_h}, we can take $H:=H_0^2(D)$ and
 $\mathcal W$ to be the law of $\psi_D$ on $\Theta:=\mathcal H^{-s}(D)$, for an arbitrary $s>(d-4)/2$. 
 (the choice of $\Theta$ is not unique as explained in \citet[Corollary 8.3.2]{Stroock:2010}). Also by theorem \ref{thm:B_0} we can index the Paley--Wiener integrals $\mathcal I(u)$ over $u\in \mathcal H_0^2(D)$ or take the maps $\mathcal I(B_0(f))$ over $f\in \mathcal H^{-2}(D)$.
\begin{remark}\label{rem:equiv_index}
By means of integration by parts we obtain, 
for every $f\in C_c^\infty(D)$, that the solution $u_f$ of the boundary value problem 
\begin{equation}\label{eq:bvprobl_class}
\begin{cases}
\Delta^2 u(x)=f(x),&x\in D\\
D^\beta u(x)=0,&|\beta|\le 1,\,x\in\partial D.
 \end{cases}
 \end{equation}
is such that for all $v\in C_c^\infty(D)$
\[\int_{D} v(x) f(x)\De x = \int_{D} v(x)\Delta^2u_f(x)\De x= \la v,\, u_f\ra_{H^2_0}.\]
Using the denseness of $C_c^\infty(D)$ in $H^2_0(D)$ we conclude from Theorem \ref{thm:B_0} that $B_0f=u_f$. Thus we have 
\[\|f\|^2_{-2} = \int_{D} u_f(x) f(x)\De x = \|u_f\|^2_{H^2_0} .\]

\end{remark}
 

\subsection{Discretisation set-up}\label{sec:set_up}
We will use the parameter $h:=1/N$ for $N\in \mathbb N$. Let $D_h:= \overline D \cap h\mathbb{Z}^d$. 
Let us denote by $R_h$ the set of points $\xi$ in $D_h$ such that for every $i,\, j\in\{1,\,\ldots\,d\}$, 
the points $\xi\pm h(e_i\pm e_j),\,\xi\pm he_i$ are all in $D_h$. 
Let $\Lambda_N= \frac1h R_h \subset \mathbb{Z}^d$ be the ``blow-up'' of $R_h$. In other words, $\Lambda_N\subset N\overline D\cap \Z^d$ is the largest set satisfying $\partial_2\Lambda_N\subset N\overline D\cap \Z^d$ where $\partial_2\Lambda_N:=\{y\in\Z^d\setminus\Lambda_N:\mathrm{dist}(y,\,\Lambda_N)\le 2\}$ is the double (outer) boundary of $\Lambda_N$ of points at $\ell^1$ distance at most $2$ from it. 
Let $(\vr_z)_{z\in \Lambda_N}$ be the membrane model 
on $\Lambda_N$ whose covariance is denoted by $G_{\Lambda_N}$. It satisfies the following boundary value problem: for all $x\in \Lambda_N$,
\begin{align}\label{eq:cov_Lambda}
\left\{\begin{array}{lr}
\Delta_1^2 G_{\Lambda_N}(x,y) = \delta_x(y), & y\in \Lambda_N\\
G_{\Lambda_N}(x,y) = 0,   &  y\notin \Lambda_N
\end{array}\right..
\end{align}
Define $\psi_h$ by 
\begin{equation}
(\psi_{h},\,f):=\kappa\sum_{x\in R_h } h^{\frac{d+4}2}\vr_{x/h} f(x) \,,\,\,\, f\in\mathcal H^s_0(D).
\end{equation}
We first show that $\psi_h \in \mathcal H^{-s}(D)$ for all $s> d/2+\floor{d/2} +1$. Clearly $\psi_h$ is a linear functional on $\mathcal H^s_0(D)$. To show $\psi_h$ is bounded, with the aid of Lemma~\ref{lem:bounds_GGS} we observe that 
\begin{align*}
\sum_{j\ge 1} \lambda_j^{-\frac s2} (\psi_h, u_j)^2&=\kappa^2h^{d+4}\sum_{j\ge 1} \lambda_j^{-\frac s2}\Bigg(\sum_{x\in R_h}\vr_{x/h}u_j(x)\Bigg)^2\\
& \overset{\eqref{bound0}}{\le} \kappa^2h^{d+4} \Bigg(\sum_{x\in R_h}|\vr_{x/h}|\Bigg)^2 \sum_{j\ge 1} \lambda_j^{-\frac s2 + 2l_0}
\end{align*}
Now using Proposition \ref{prop:Weyl} we conclude that the sum in the right hand side in finite whenever $s> d/2+\floor{d/2} +1$. Thus we have shown that $\psi_h \in \mathcal H^{-s}(D)$ for all $s> d/2+\floor{d/2} +1$ and we have
\begin{align}\label{psi_h_norm}
\|\psi_h \|_{-s}^2 = \sum_{j\ge 1} \lambda_j^{-\frac s2} (\psi_h, u_j)^2.
\end{align}
The result we want to show is 
\begin{theorem}[Scaling limit in $d\ge 4$]\label{thm:critical_d}
One has that, as $h\to 0$, the field $\psi_h$ converges in distribution to $\psi_D$ of Proposition~\ref{prop:series_rep_h} in the topology of $\mathcal H^{-s}(D)$ for $s>s_d$, where
	$$s_d:=\frac{d}{2} + 2\left(\ceil*{ \frac{1}{4}\left(\floor*{\frac{d}{2}}+1\right)} + \ceil*{ \frac{1}{4}\left(\floor*{\frac{d}{2}}+6\right)} -1\right).$$
\end{theorem}
\begin{remark}
	An analogous result holds in $d=2,\,3$, but we will not discuss it here as it is superseded by Theorem~\ref{thm:low_d}. 
\end{remark}
\subsection{Proof of the scaling limit (Theorem~\ref{thm:critical_d})}
Once again we need to prove tightness and ``convergence of marginal laws''. In $d\ge 4$ however we are concerned with a field which is not defined 
pointwise, so that ``marginal'' from now takes on the meaning of the law of $(\psi_h,\,f)$, namely the action of $\psi_h$, 
seen as a distribution, on 
the test function $f$. The results are built on the approximation of the continuum Dirichlet problem for the bilaplacian by \cite{thomee}, combined with classical 
embeddings for Sobolev spaces.
\subsubsection{Convergence of the marginals}\label{sec:critical_d}
To prove that the scaling limit is indeed $\psi_D$ we first have to find the marginal limiting laws. The set $C_c^\infty(D)$ is dense in $\mathcal H^{s}_0(D)$, so we can use only smooth and compactly 
supported functions to test the convergence.
\begin{proposition}\label{f1}
	$(\psi_{h},\,f)$ converges in law to $(\psi_D,\,f)$ as $h\to 0$ for any $f$ smooth and compactly supported in $D.$ 
\end{proposition}
\begin{proof}
	Since the Gaussian field $\varphi$ is centered, we shall focus on the convergence of the variance only. Note that $\var\left(\psi_D,\,f\right)=\|f\|^2_{-2}$. Remark~\ref{rem:equiv_index} tells us that 
	we can limit ourselves to showing that
	\[
	\lim_{h\to 0}\var( \psi_h,\, f)= \int_{D} u(x) f(x)\De x
	\]
	where $u$ is the solution of \eqref{eq:bvprobl_class}. 
	We define 
	$$G_{R_h}(x,y):=\E[ \vr_{x/h}\vr_{y/h}]\,,\,\,x,y \in D_h.$$
	Note that if $\Delta_h$ (defined in Appendix~\ref{appendix:T}) is the discrete Laplacian on $h\Z^d$ then by ~\eqref{eq:cov_Lambda} we have, for all $x\in R_h$,
	\[ \left\{\begin{array}{lr}
	\Delta_h^2 G_{R_h}(x,y) = \frac{4d^2}{h^4}\delta_{x}(y), & y\in R_h \\
	G_{R_h}(x,y) = 0 ,& y\notin R_h
	\end{array}\label{eq:G_h}\right..
	\]
	We have
	\begin{align*}
	\var[(\psi_h, f)]&= \kappa^2\sum_{x, y\in R_h} h^{d+4} G_{R_h}(x,y) f(x) f(y)\\
	&=\sum_{x\in R_h} h^d H_h(x) f(x)
	\end{align*}
	where $H_h(x)=\kappa^2 \sum_{y\in R_h}h^4 G_{R_h}(x,y) f(y)$, $x\in D_h$. It is immediate that $H_h$ is the solution of the following Dirichlet problem,
	\[\begin{cases}
	\Delta_h^2 H_h(x) = f(x),& \quad x\in R_h\\
	H_h(x)= 0, &\quad x\notin R_h.
	\end{cases}\]
	It is known that the above discrete solution is close to the continuum solution. The details of the result are described in 
	Appendix~\ref{appendix:T}; 
	here we only recall that if we define $e_h(x):= u(x)-H_h(x)$ for $x\in D_h$ and $R_h f$ is the restriction of a function $f$ to the set 
	$R_h$ as in~\eqref{eq:R_h}, then from Theorem \ref{thm:one} we have
	\begin{equation}\label{eq:thomee2}
	\|R_he_h\|_{h,\,grid} \le Ch^{1/2}.
	\end{equation}
	We have defined $\|f\|_{h,\,grid}^2:=h^d\sum_{\xi\in h\Z^d}f(\xi)^2$, where $f$ is any grid function with finite support.
	Hence we get that
	$$\var[(\psi_h, f)]= -\sum_{x\in R_h} e_h(x) f(x)h^d + \sum_{x\in R_h} u(x) f(x) h^d.$$
	Note that by Cauchy--Schwarz the first term in absolute value is bounded by $ \|R_he_h\|_{h,\,grid}\|f\|_{h,\,grid}$ and it goes to zero by~\eqref{eq:thomee2} as $h\to 0$. 
	For the second term we have
	\begin{equation}\label{eq:limit1}
	\lim_{h\to 0}\sum_{x\in R_h} u(x) f(x) h^d= \int_{D} u(x) f(x)\De x.
	\end{equation} 
\end{proof}
\subsubsection{Tightness}
%
We next prove the following lemma.
\begin{lemma}\label{lem:limsup_psi}
	$$\limsup_{h\to 0}\E[\|\psi_h\|_{-s}^2]<\infty \quad \forall \, s>s_d.$$
\end{lemma}
\begin{proof}
	From \eqref{psi_h_norm} we have
	$$\E\left[\|\psi_h\|_{-s}^2\right] =\sum_{j\in \N}\lambda_j^{-s/2}\E[(\psi_h\,,\,u_j)^2].$$
	Note that $u=\lambda_j^{-1}u_j$ is the unique solution of \eqref{eq:bvprobl_class} for $f=u_j$. We therefore obtain as in the proof of proposition \ref{f1} by defining $e_{h,j}$ to be the error corresponding to $f=u_j$
	\begin{align*}
	\E[(\psi_h\,,\,u_j)^2]&= -\sum_{x\in R_h} e_{h,j}(x) u_j(x)h^d + \sum_{x\in R_h} \lambda_j^{-1}u_j(x) u_j(x) h^d\\
	& \le C\sup_{x\in D}|u_j(x)| \left({h^d \sum_{x\in R_h} e_{h,j}(x)^2}\right)^{1/2}+ C\lambda_j^{-1}\left(\sup_{x\in D}|u_j(x)|\right)^2.
	\end{align*}
	Using Theorem \ref{thm:one} along with the bounds~\eqref{bound0}-\eqref{bound1}-\eqref{bound2} we obtain
	\begin{align*}
	\E[(\psi_h\,,\,u_j)^2]&\le C \lambda_j^{l_0} [\lambda_j^{2l_5-2} h^2 + h\left(\lambda_j^{2l_5-2} h^6 + \lambda_j^{2l_2-2}\right)]^{\frac12} + C\lambda_j^{2l_0-1}\\
	&\le C \lambda_j^{l_0 + l_5-1}.
	\end{align*}	
	Therefore we have 
	\begin{align*}
	\E\left[\|\psi_h\|_{-s}^2\right]\le C \sum_{j\in \N}\lambda_j^{-\frac s2}\lambda_j^{l_0 + l_5-1}.
	\end{align*}
	Thus 
	$$\limsup_{h\to 0}\E[\|\psi_h\|_{-s}^2]<\infty \qquad \text{ if } \qquad \sum_{j\in \N}\lambda_j^{-\frac s2+l_0 + l_5-1}<\infty.$$ And from proposition \ref{prop:Weyl} we obtain that $\sum_{j\in \N}\lambda_j^{-\frac s2+l_0 + l_5-1}<\infty$ whenever $s>s_d$.
\end{proof}

To show tightness of $\psi_h$ we need the following theorem:
 \begin{theorem}\label{Rel:thm}
 	For $0\le s_1<s_2$, $\mathcal H^{-s_1}(D)$ is compactly embedded in $\mathcal H^{-s_2}(D)$.
 \end{theorem}
 \begin{proof}
 	It is enough to prove that $\mathcal H_0^{s_2}(D)$ is compactly embedded in $\mathcal H_0^{s_1}(D)$. 
 	The inclusion $\mathcal H_0^{s_2}(D)\hookrightarrow\mathcal H_0^{s_1}(D)$ is linear and 
 	continuous. To prove the inclusion to be compact let $B$ be the unit ball of 
 	$\mathcal H_0^{s_2}(D)$. Given $\epsilon>0$ we choose $N\in\mathbb{N}$ large enough so 
 	that $N^{s_1-s_2}<\epsilon^4$. Now we consider the subspace $Z$ of $\mathcal H_0^{s_2}(D)$ 
 	defined by $Z:=\left\lbrace f\in\mathcal H_0^{s_2}(D):(f,\,u_j)_{L^2}=0\,\,\forall\,j<N\right\rbrace$. 
 	Then for any $f\in B\cap Z$ we have
     \begin{align*}
     \lVert f\rVert_{s_1}^2&= \sum_{j\in \N}\lambda_j^{s_1/2}(f,u_j)_{L^2}^2=\sum_{j\ge N}\lambda_j^{s_1/2}(f,u_j)_{L^2}^2=\sum_{j\ge N}\lambda_j^{s_1/2-s_2/2}\lambda_j^{s_2/2}(f,u_j)_{L^2}^2\\
     &\le N^{(s_1-s_2)/2}\sum_{j\ge N}\lambda_j^{s_2/2}(f,u_j)_{L^2}^2= N^{(s_1-s_2)/2}\lVert f\rVert_{s_2}^2 < \epsilon^2.
     \end{align*}
     Also note that the dimension of $\mathcal H_0^{s_2}(D)/Z$ is finite, so the unit ball of $\mathcal H_0^{s_2}(D)/Z$ is compact and hence can be covered by finitely many balls of radius $\epsilon$. Hence $B$ can be covered by finitely many balls of radius $2\epsilon$ in the $\lVert \cdot\rVert_{s_1}$-norm. Since $\epsilon$ is arbitrary, $B$ is precompact in $\mathcal H_0^{s_1}(D)$. Therefore the inclusion map is compact.
 \end{proof}
 \begin{corollary}
 	The sequence $(\psi_h)_{h=\frac1N, N\in \mathbb{N}}$ is tight in $\mathcal H^{-s}(D)$ for all $s>s_d$.
 \end{corollary}
 \begin{proof}
 	Fix $s_0 > s_d$ and let $s_d < s_1 < s_0$. By Theorem \ref{Rel:thm}, for any $R > 0$, $\overline{B_{\mathcal H^{-s_1}(D)}(0,\,R)}$ is compact in $\mathcal H^{-s_0}(D)$. 
 	By Lemma~\ref{lem:limsup_psi} we have for some $M > 0$
 	$$\E[\|\psi_h\|_{-s_1}^2] \le M \quad \forall\,h.$$
 	Given $\epsilon > 0$, we take $R=\sqrt{{2M}{\epsilon}^{-1}}$ so that ${M}R^{-2} < \epsilon$. Now for all $h$
 	\begin{align*}
 	\prob\left(\psi_h \notin \overline{B_{\mathcal H^{-s_1}(D)}(0\,,\,R)}\right)&= \prob\left(\|\psi_h\|_{-s_1} > R\right)\le 
 	\frac{\E[\|\psi_h\|_{-s_1}^2]}{R^2}< \epsilon .
 	\end{align*}
 	Thus $(\psi_h)_h$ is tight in $\mathcal H^{-s_0}(D).$ 
 \end{proof}
Having obtained tightness and convergence of the marginals, 
 all is left to do is to combine these ideas together to show the scaling limit.
\begin{proof}[Proof of Theorem~\ref{thm:critical_d}]
  As $(\psi_h)$ is tight in $\mathcal H^{-s}(D)$, 
  it is enough to prove that every converging subsequence $(\psi_{h_i})$ 
  converges in distribution to $\psi_D$. Let $(\psi_{h_i})$ be a 
  subsequence of $(\psi_h)$ converging in distribution to $\psi$ in $\mathcal H^{-s}(D)$. 
  Then $( \psi_{h_i}, f)$ converges in distribution to $(\psi , f )$ 
  for any $f\in \mathcal H_0^s(D)$. But since $(\psi_h, f)$ converges in distribution to $( \psi_D , f )$ for all $f\in C_c^\infty(D)$, we must have $( \psi_D , f) \overset{d}{=}( \psi , f)$ for all $f\in C_c^\infty(D)$. Now let $g\in \mathcal H_0^s(D)$. 
  Since $C_c^\infty(D)$ is dense in $\mathcal H_0^s(D)$ we have a sequence $(f_k)$ in $C_c^\infty(D)$ 
  such that $f_k\rightarrow g$ in $\mathcal H_0^s(D)$. 
  Therefore $( \psi_D , f_k )$ and $(\psi , f_k)$ converge 
  to $(\psi_D , g )$ and $( \psi , g)$ respectively. 
  And hence $(\psi_D , f_k)$ and $(\psi , f_k)$ converge in distribution 
  to $( \psi_D , g)$ and $( \psi , g )$ respectively. 
  But since $( \psi_D , f_k ) \overset{d}{=}( \psi , f_k)$ for all $k$, 
  we have $(\psi_D , g ) \overset{d}{=}( \psi , g )$. Thus we 
  have $( \psi_D , f )\overset{d}{=}(\psi , f )$ for all $f\in \mathcal H_0^s(D)$. 
  Hence $\psi_D \overset{d}{=}\psi$, since the fields under considerations are linear.
 \end{proof}

%
%
%
%
 \section{Convergence in infinite volume in \texorpdfstring{$d\geq 5$}{}}\label{sec:big_d}
 \subsection{Description of the limiting field}
In this section we deal with the infinite volume membrane model defined on the whole of $\Z^d$ and show that the rescaled field 
 converges to the continuum bilaplacian field on $\R^d$. Let $\prob_N$ be the finite volume MM measure defined on $V_N$ as mentioned in the Introduction. It is known that in $d\ge 5$ there exists $\prob$ on $\R^{\Z^d}$ such that $\prob_N\to \prob$ in the weak topology of probability measures (\citet[Proposition 1.2.3] {Kurt_thesis}).  Under $\prob$, the canonical coordinates $(\vr_x)_{x\in \Z^d}$ form a centered Gaussian process with covariance given by
$$G(x,y)= \Delta^{-2}(x,\, y)= \sum_{z\in \Z^d} \Delta^{-1}(x,z)\Delta^{-1}(z,\, y)= \sum_{z\in \Z^d} \Gamma(x,z) \Gamma(z,y),$$
where $\Gamma$ denotes the covariance of the DGFF. $\Gamma$ has an easy representation in terms of the simple random walk $(S_n)_{n\ge 0}$ on $\Z^d$ given by
$$\Gamma(x, \,y)=\sum_{m\ge 0} \mathrm P_x[ S_m=y]$$
($\mathrm P_x$ is the law of $S$ starting at $x$). This entails that
\begin{equation}\label{eq:cov_membrane}
G(x,\,y)= \sum_{m\ge 0} (m+1) \mathrm P_x[ S_m=y]=\mathrm E_{x, y }\left[ \sum_{\ell,\,m=0}^{+\infty} \one_{\left\{S_m= \tilde S_\ell\right\} }\right]
\end{equation}
where $S$ and $\tilde S$ are two independent simple random walks started at $x$ and $y$ respectively.  
First one can note from this representation that $G(\cdot,\,\cdot)$ is translation invariant. The existence of the infinite volume measure in $d\ge 5$ gives that $G(0,0)<+\infty$. Using the above one can derive the following property of the covariance:
\begin{fact}[{\citet[Lemma 5.1]{Sakagawa}}]\label{fact: covariance:mm}
\begin{equation}\label{eq:cov:mm}
\lim_{\|x\|\to+\infty}\frac{G(x,0)}{\|x\|^{4-d}}=\eta_2
\end{equation}
where
$$
\eta_2=(2\pi)^{-d}\int_0^{+\infty}\int_{\R^d}\exp\left(\iota\langle \zeta,\, \theta\rangle-\frac{\|\theta\|^4 t}{4\pi^2}\right)\De\theta\De t
$$
for any $\zeta\in \mathbb{S}^{d-1}$.
\end{fact}
 It is convenient to consider the convergence in the space of tempered distribution (dual of the Schwartz space on $\R^d$). For this we are giving some 
 preliminary theoretical results.
 
\subsubsection{Generalized random fields and limiting field}

We consider $\mathcal S=\mathcal S(\mathbb{R}^d)$ to be the Schwartz space that consists of infinitely differentiable functions 
$f: \mathbb{R}^d\rightarrow \mathbb{R}$ such that, for all $m\in\mathbb{N}\cup\{0\}$ and $\alpha=(\alpha_1,\ldots,\alpha_d)\in(\mathbb{N}\cup\{0\})^d$,
$$\lVert f\rVert_{m,\alpha}= \sup_{x\in\mathbb{R}^d }(1+\lVert x\rVert^m)\lvert D^\alpha f(x)\rvert <\infty .$$
$\mathcal S$ is a linear vector space and it is equipped with the topology generated by the family of semi-norms $\lVert \cdot\rVert_{m,\alpha},$ $ m\in \mathbb{N}\cup\{0\}$ and $\alpha\in(\mathbb{N}\cup\{0\})^d$. The topological dual $\mathcal S^{*}$ of $\mathcal S$ is called the space of tempered distributions. For $F\in \mathcal S^{*}$ and $f\in \mathcal S$ we denote $F(f)$ by $(F,f)$. We shall work with two topologies on $\mathcal S^{*}$, the strong topology $\tau_s$ and the weak topology $\tau_w$. The strong topology $\tau_s$ is generated by the family of semi-norms $\left\lbrace e_B: B \text{ is a bounded subset of } \mathcal S \right\rbrace $ where $e_B(F)=\sup_{f\in B}(F,f) , F\in \mathcal S^{*}$. 
$\tau_w$ is induced by the family of semi-norms $\left\lbrace \lvert(\cdot\,,f)\rvert: f\in\mathcal S\right\rbrace $. In particular $F_n$ converges to $F$ in $\mathcal S^{*}$ 
with respect to the weak topology when $\lim_n (F_n,f)=(F,f)$ for all $ f\in \mathcal S$. 
It can be shown that the Borel $\sigma$-fields corresponding to both topologies coincide. Therefore we shall talk about the Borel $\sigma$-field $\mathcal B(\mathcal S^{*})$ of $\mathcal S^{*}$ without specifying the topology.

Let $(\Omega\,, \mathcal A,\,\mathrm P)$ be a probability space. By a generalized random field defined on $(\Omega, \mathcal A,\,\mathrm P)$, 
we refer to a random variable $X$ with values in $(\mathcal S^{*},\mathcal B(\mathcal S^{*}))$. For 
$(X_n)_{n\ge 1}$ and $X$ generalized random fields with laws $(\prob_{X_n})_{n\ge 1}$ and $\prob_X$ respectively, we say that 
$X_n$ converges in distribution to $X$ (and write $X_n \overset{d}\to X$) 
with respect to the strong topology if 
$$\lim_{n\to\infty}\int_{\mathcal S^{*}}\varphi(F)\mathrm{d}\prob_{X_n}(F) = \int_{\mathcal S^{*}}\varphi(F)\mathrm{d}\prob_{X}(F) \quad\forall\,\varphi\in  C_b(\mathcal S^{*}, \tau_s)$$
where $ C_b(\mathcal S^{*}, \tau_s)$ is the space of bounded continuous 
functions on $\mathcal S^{*}$ given the strong topology. The convergence in distribution with 
respect to the weak topology is defined similarly with test functions in $ C_b(\mathcal S^{*}, \tau_w)$. For a generalized random field $X$ with law $\prob_X$, we define its 
characteristic functional by 
$$ \mathcal L_X(f)=\E(\e^{\iota (X, f)})=\int_{\mathcal S^{*}}\e^{\iota (F, f)}\mathrm{d}\prob_X(F) $$
for $f\in \mathcal S$. Note that $\mathcal L_X$ is positive definite, continuous, and $\mathcal L_X(0)=1$. 
The Bochner--Minlos theorem says that the converse is also true: if a functional $\mathcal L: \mathcal S \rightarrow \mathbb{C}$ is 
positive definite, continuous at $0$ and satisfies $\mathcal L(0)=1$ then there exists a 
generalized random field $X$ defined on a probability space $(\Omega\,, \mathcal A,\,\mathrm P)$ such that $\mathcal L_X=\mathcal L$. For a proof of this theorem see 
for instance \citet[Appendix~1]{Hida:2004}. Another important feature of characteristic functions is that 
their convergence determines convergence of generalised random fields. 
This is classical result of L\'evy which was generalized and proved in the nuclear space setting first by \cite{Fernique:1968}. 
We use the version for tempered distributions which was recently proved in \cite{BOY2017}. 
\begin{fact}[{\citet[Corollary 2.4]{BOY2017}}] \label{equiv:fact}
	Let $(X_n)_{n\ge 1},\,X$ be generalized random fields. The following conditions are equivalent:
	\begin{enumerate}
		\item $X_n \overset{d}\to X$ in the strong topology.
		\item $X_n \overset{d}\to X$ in the weak topology.
		\item $\mathcal L_{X_n}(f) \to \mathcal L_{X}(f)$ for all $f\in \mathcal S$.
		\item $(X_n, f) \overset{d}\to (X, f)$ in $\mathbb{R}$  for all $f\in \mathcal S$.
	\end{enumerate}	
\end{fact} 

For $f\in \mathcal S$ we define $\widehat f\in \mathcal S$ by $$\widehat f(\theta)=\frac1{(2\pi)^{d/2}}\int_{\R^d}\e^{-\iota\la x, \theta \ra} f(x)\De x.$$ 
Let us define an operator $(-\Delta)^{-1} : \mathcal S\to L^2(\R^d)$ as follows \cite[Section~1.2.2]{Adams/Hedberg:2012}:
$$(-\Delta)^{-1}f(x):= \frac1{(2\pi)^{d/2}}\int_{\R^d} \e^{\iota \la x, \xi\ra } \|\xi\|^{-2} \widehat f(\xi) \De\xi.$$
We use now the operator $(-\Delta)^{-1}$ to define the limiting field $\psi$. It is the fractional Gaussian field of 
parameter $s:=2$ described in \citet[Section~3.1]{LSSW}, to 
which we refer for a proof of 
the following fact, relying on the Bochner--Minlos theorem.
\begin{lemma}\label{lemma:psi}
There exists a generalized random field $\psi$ on $\mathcal S^*$ whose characteristic functional 
$\mathcal L_\psi$ is given by
	$$\mathcal L_{\psi} (f)= \exp\left(- \frac12\lVert (-\Delta)^{-1}f\rVert^2_{L^2(\R^d)}\right),\quad f\in \mathcal S.$$
\end{lemma}

Consider $(\vr_x)_{x\in\Z^d}$ to be the membrane model in $d\ge 5$. We define $$\psi_N(x):= \kappa N^{\frac{d-4}{2}} \vr_{Nx}, \quad x\in \frac1N \Z^d.$$
For $f\in \mathcal S$ we define 
\begin{equation}\label{def:fieldS}
\left( \psi_N, f\right): = N^{-d} \sum_{x\in \frac1N \Z^d} \psi_N(x) f(x).
\end{equation}
The above definition makes sense since, using Mill's ratio and the uniform boundedness of $G(\cdot,\,\cdot)$, one can show that, as $\|x\|\to\infty$,
\[
 |\psi_N(x)|=O\big(N^{\frac{d-4}{2}}\sqrt{\log(1+\|x\|)}\big)\quad a.s.
\]
via a Borell-Cantelli argument. This justifies~\eqref{def:fieldS} using the fast decay of $f$ at infinity. Also it follows that $\psi_N \in \mathcal S^*$ and the characteristic functional of $\psi_N$ is given by $$\mathcal L_{\psi_N}(f):= \exp(- \var\left(\psi_N,f\right)/2).$$

The following Theorem shows that the field $\psi_N$ constructed above converges to $\psi$ defined in Lemma~\ref{lemma:psi}.
\begin{theorem}[Scaling limit in $d\ge 5$]\label{thm:big_d}
Let $d\ge 5$ and $\psi_N$ be the field on $\mathcal S^*$ defined by~\eqref{def:fieldS}. Then  $\psi_N\overset{d}\to \psi$ in the strong topology where $\psi$ is defined in Lemma~\ref{lemma:psi}.
\end{theorem}
\subsection{Proof of the scaling limit (Theorem~\ref{thm:big_d})}
The proof of our last Theorem relies on the result recalled in Fact~\ref{equiv:fact}, therefore unlike the two previous Theorems 
it is not divided into tightness and finite dimensional convergence. The argument is based on Fourier analysis, and will be a consequence of 
two claims which we will show after the main proof.
\begin{proof}[Proof of Theorem~\ref{thm:big_d}]We first show that for any $f\in \mathcal S(\mathbb{R}^d)$,
	$$\E\left[\left( \psi_N, f\right)^2\right]\rightarrow \|(-\Delta)^{-1}f\|^2_{L^2(\R^d)}.$$
By our definition we have for $f,\,g\in\mathcal S$ 
$$\mathrm{Cov}(\left( \psi_N, f\right),\left( \psi_N, g\right))=\kappa^2N^{-(d+4)}\sum_{x,y\in \frac1N\Z^d } G(0, N(y-x))f(x) g(y).$$
Hence
	\begin{align*}
	\E\left[ \left( \psi_N,  f\right)^2 \right]=\kappa^2N^{-(d+4)} \sum_{x,y\in \frac1N \Z^d} G(0, N(y-x))f(x) f(y).
	\end{align*}
We deduce from the Fourier inversion formula, in the same fashion of the proofs of \citet[Lemmas~1.2.2,~1.2.3]{Kurt_thesis}, that
	$$G(0,x)= \frac1{(2\pi)^d} \int_{[-\pi, \pi]^d} \left(\mu(\theta)\right)^{-2}\e^{-\iota\la x, \theta\ra} \De \theta$$	
where $\mu(\theta)=\frac1{d}\sum_{i=1}^d (1-\cos(\theta_i))=\frac2d\sum_{i=1}^d \sin^2(\frac{\theta_i}{2})$. Hence we have
	\begin{align}
	\E\left[ \left( \psi_N, f\right)^2 \right]&=\frac{ \kappa^2N^{-(d+4)}}{(2\pi)^d} \sum_{x,y\in  \frac1N \Z^d}\int_{[-\pi, \pi]^d} \left(\mu(\theta)\right)^{-2} \e^{-\iota\la N(y-x), \theta\ra} f(x) f(y) \De\theta\nonumber\\
	&=\frac{ \kappa^2N^{-(d+4)}}{(2\pi)^d} \sum_{x,y\in  \frac1N \Z^d} \int_{[-\pi, \pi]^d} \left(\mu(\theta)\right)^{-2} \e^{-\iota\la (y-x), N\theta\ra} f(x) f(y) \De\theta\nonumber\\
	&=\frac{ \kappa^2N^{-4}}{(2\pi)^d}\int_{[-N\pi, N\pi]^d}  \left(\mu\left(\frac{\theta}{N}\right)\right)^{-2} \left|N^{-d} \sum_{x\in  \frac1N \Z^d} \e^{-\iota\la x, \theta \ra} f(x)\right|^2\De \theta.
	\end{align}
	We have used in the above Fubini's theorem, justified by the following bound \cite[Lemma~7]{CHR}:
	there exists $C>0$ such that for all $N\in \N$ and $w\in [-N\pi/2, N\pi/2]^d\setminus \{0\}$  we have
	\begin{align}
	\frac{1}{\| w\|^4}\le N^{-4}\left(\sum_{i=1}^d \sin^2\left(\frac{w_i}{N}\right)\right)^{-2}\le \left(\frac{1}{\|w\|^2}+\frac{C}{N^{2}}\right)^2\label{eq:sine's_bound}.
	\end{align}
We make two claims which will prove the convergence of variance.
\begin{claim}\label{claim:error1}
	\begin{align*}
	\lim_{N\to+\infty} &\left|\frac{ \kappa^2N^{-4}}{(2\pi)^d}\int_{[-N\pi, N\pi]^d}  \left(\mu\left(\frac{\theta}{N}\right)\right)^{-2} 
	\Bigg|N^{-d} \sum_{x\in  \frac1N \Z^d} \e^{-\iota\la x, \theta \ra} f(x)\Bigg|^2\De \theta\right.\\
	&\qquad  \qquad \qquad \left. - \frac{1}{(2\pi)^d}\int_{[-N\pi, N\pi]^d} \|\theta\|^{-4} \Bigg|N^{-d} \sum_{x\in  \frac1N \Z^d} \e^{-\iota\la x, \theta \ra} f(x)\right|^2\De \theta\Bigg|=0.
	\end{align*}
	\end{claim}
	
Next we claim the convergence of the following term:
\begin{claim}\label{claim:RS}
	$$\lim_{N\to+\infty}\frac1{(2\pi)^{d}}\int_{[-N\pi, N\pi]^d}\|\theta\|^{-4}\Bigg|N^{-d} \sum_{x\in  \frac1N \Z^d} \e^{-\iota\la x, \theta \ra} f(x)\Bigg|^2\De \theta=\|(-\Delta)^{-1}f\|^2_{L^2(\R^d)}.$$
	\end{claim}
	Claims~\ref{claim:error1}-\ref{claim:RS} entail that
	$$\lim_{N\to\infty}\mathcal L_{\psi_N}(f)=\exp\left(-\frac12 
	\|(-\Delta)^{-1}f\|^2_{L^2(\R^d)}\right).$$
	Thus we have for all $f\in \mathcal S$
	$$\mathcal L_{\psi_N}(f)\to \mathcal L_{\psi}(f).$$
	Hence the conclusion follows from Fact \ref{equiv:fact}. 
	\end{proof}
	
To prove the above two claims we use crucially the following estimate for approximating Riemann sums for Schwartz functions. 
Since we could not find a reference we provide a short proof of the following fact:
\begin{lemma} For any $N\ge 1$ and $s>0$ we have 
\begin{equation}\label{eq:decayRS}
\left|(2\pi)^{-d/2} N^{-d}\sum_{x\in \Z^d} \e^{-\iota \la \frac{x}{N}, \theta\ra} f\left(\frac{x}{N}\right)-\widehat f(\theta)\right|\le CN^{-s}
\end{equation}
where $C$ may depend on $f$. 
\end{lemma}
\begin{proof}
To show the above result we use the Poisson summation formula \cite[Chapter 7]{Stein:Weiss}. 
Let us define $g(x):= (2\pi)^{-d/2}\e^{-\iota \la x, \theta\ra} f\left(x\right)$. Using the Poisson summation formula we get
$$N^{-d} \sum_{x\in \Z^d} g\left(\frac{x}{N}\right)= \sum_{x\in \Z^d} \widehat f(\theta+2\pi xN).$$
Hence we have
\begin{align*}
\left| (2\pi)^{-d/2}N^{-d}\sum_{x\in \Z^d} \e^{-\iota \la \frac{x}{N}, \theta\ra} f\left(\frac{x}{N}\right)-\widehat f(\theta)\right|&\le \sum_{x\neq 0, x\in \Z^d} \lvert\widehat f(\theta+2\pi xN)\rvert\le \sum_{x\neq 0, x\in \Z^d}\frac{C}{\|\theta+2\pi xN\|_{\infty}^s}
\end{align*}
where the last inequality holds for any $s\ge 0$ because $\widehat f\in \mathcal S$. But
\[\|2\pi xN\|_{\infty}\le \|\theta+2\pi xN\|_{\infty}+\|\theta\|_{\infty}\le \|\theta+2\pi xN\|_{\infty}+N\pi\]
and hence, for $s>1$, $\|2\pi xN\|_{\infty}^s\le 2^{s-1} \left( \|\theta+2\pi xN\|_{\infty}^s+ (N\pi)^s\right).$ Thus for any $s\ge d_0>d$, we have
\[\left|(2\pi)^{-d/2} N^{-d}\sum_{x\in \Z^d} \e^{-\iota \la \frac{x}{N}, \theta\ra} f\left(\frac{x}{N}\right)-\widehat f(\theta)\right|\le \sum_{x\neq 0, x\in \Z^d}\frac{C}{(N\pi)^s(2\|x\|_{\infty}^s-1)} \le CN^{-s}.\qedhere\]
where the constant $C$ depends on $d_0$ but not on $s$. Hence the result follows.
\end{proof}
We can now begin with the proof of the two claims.
\begin{proof}[Proof of Claim~\ref{claim:error1}] Recall that $\kappa=1/(2d)$. Using the bound \eqref{eq:sine's_bound} for $w_i=\theta_i/2$ we have 
\begin{align*}
&\left|\frac{ \kappa^2N^{-4}}{(2\pi)^d}\int_{[-N\pi, N\pi]^d} \left(\mu\left(\frac{\theta}{N}\right)\right)^{-2} |N^{-d} \sum_{x\in  \frac1N \Z^d} \e^{-\iota\la x, \theta \ra} f(x)|^2\De \theta\right.\\
&\qquad \qquad \qquad \left. - \frac{1}{(2\pi)^d}\int_{[-N\pi, N\pi]^d} \|\theta\|^{-4} \left|N^{-d} \sum_{x\in  \frac1N \Z^d} \e^{-\iota\la x, \theta \ra} f(x)\right|^2\De \theta\right|\\
&\le \int_{[-N\pi, N\pi]^d} \left(2\|\theta\|^{-2}\frac{C}{N^2}+\frac{C}{N^4}\right) \left|(2\pi)^{-d/2}N^{-d} \sum_{x\in  \frac1N \Z^d} \e^{-\iota\la x, \theta \ra} f(x)\right|^2\De \theta.
\end{align*}
Using $(a+b)^2\le 2(a^2+b^2)$ after adding and subtracting $\widehat f(\theta)$ in the modulus above we have the bound
\begin{align*}
 &\int_{[-N\pi, N\pi]^d} \left(2\|\theta\|^{-2}\frac{C}{N^2}+\frac{C}{N^4}\right)(CN^{-s}+|\widehat{f}(\theta)|)^2\De \theta\\
 &\qquad \le \int_{[-N\pi, N\pi]^d}2\left(2\|\theta\|^{-2}\frac{C}{N^2}+\frac{C}{N^4}\right)(CN^{-2s}+|\widehat{f}(\theta)|^2)\De \theta.
 \end{align*}
 Again the last term amounts to estimating
\begin{align*}
&CN^{-2s-2}O(N^{d-2})+CN^{-2s-4}O(N^d)\\
&+CN^{-2}\int_{[-N\pi, N\pi]^d}\|\theta\|^{-2}|\widehat{f}(\theta)|^2\De\theta+CN^{-4}\int_{[-N\pi, N\pi]^d}|\widehat{f}(\theta)|^2 \De\theta
\end{align*}
which goes to $0$ due to the fact that $f\in \mathcal S$. 
\end{proof}
\begin{proof}[Proof of Claim~\ref{claim:RS}] We have $$\|(-\Delta)^{-1}f\|^2_{L^2(\R^d)}= \int_{\R^d} \|\theta\|^{-4} |\widehat f(\theta)|^2 \De \theta$$ and
\begin{align*}
&\left|\int_{[-N\pi, N\pi]^d}\|\theta\|^{-4}|(2\pi)^{-d/2}N^{-d} \sum_{x\in  \frac1N \Z^d} \e^{-\iota\la x, \theta \ra} f(x)|^2- \int_{\R^d}\|\theta\|^{-4}|\widehat f(\theta)|^2\De \theta\right|\\
&\le \left|\int_{[-N\pi, N\pi]^d}\|\theta\|^{-4}|(2\pi)^{-d/2}N^{-d} \sum_{x\in  \frac1N \Z^d} \e^{-\iota\la x, \theta \ra} f(x)|^2
-\int_{[-N\pi, N\pi]^d}\|\theta\|^{-4}|\widehat f(\theta)|^2\De \theta\right|\\
&+\left|\int_{[-N\pi, N\pi]^d}\|\theta\|^{-4}|\widehat f(\theta)|^2\De \theta-\int_{\R^d}\|\theta\|^{-4}|\widehat f(\theta)|^2\De \theta\right|.
\end{align*}
Clearly the second term goes to zero as $N$ tends to infinity. As for the first term we have the following bound
\begin{align*}
&\left|\int_{[-N\pi, N\pi]^d}\|\theta\|^{-4}|(2\pi)^{-d/2}N^{-d} \sum_{x\in  \frac1N \Z^d} \e^{-\iota\la x, \theta \ra} f(x)|^2-\int_{[-N\pi, N\pi]^d}\|\theta\|^{-4}|\widehat f(\theta)|^2\De \theta\right|\\
&\le \int_{[-N\pi, N\pi]^d}\|\theta\|^{-4}\left| |(2\pi)^{-d/2}N^{-d} \sum_{x\in  \frac1N \Z^d} \e^{-\iota\la x, \theta \ra} f(x)|^2-|\widehat f(\theta)|^2\right|\De \theta\\
&\le (2\|f\|_{L^1}+ C)\int_{[-N\pi, N\pi]^d} \|\theta\|^{-4} \left|(2\pi)^{-d/2} N^{-d}\sum_{x\in \Z^d} \e^{-\iota \la \frac{x}{N}, \theta\ra} f\left(\frac{x}{N}\right)-\widehat f(\theta)\right|\De \theta\\
&=O(N^{d-4-s}).
\end{align*}
where the bound in the second inequality is obtained using the formula $(a^2-b^2)=(a+b)(a-b)$ and ~\eqref{eq:decayRS}. Thus the first term also goes to zero as $N$ tends to infinity.	\end{proof}
\appendix
 \section{Quantitative estimate on the discrete approximation in \texorpdfstring{\cite{thomee}}{}}\label{appendix:T}
This section is devoted to obtaining quantitative estimates on approximation of solutions of PDEs. The building block of our analysis is the paper \cite{thomee}. Let $V$ be any bounded domain in $\R^d$ with  $C^2$ boundary. We denote $L:=\Delta^2$ and consider the following continuum Dirichlet problem:
 	\begin{equation}\label{eqa:continuum}
 	\begin{cases}
 	Lu(x) = f(x),& x\in V\\
 	D^\beta u(x)=0,& \lvert\beta\rvert\leq 1,\,x\in \partial V. 
 	\end{cases}
 	\end{equation} 	                  
%
 	Let $h>0$. We will call the points in $h\mathbb{Z}^d$ the grid points 
 	in $\mathbb{R}^d$. We consider $L_hu:=\Delta^2_hu$ to be the discrete 
 	approximation of $Lu$, where $\Delta_h$ is defined by
 	$$\Delta_hf(x):=\frac1{h^2}\sum_{i=1}^d (f(x+he_i)+f(x-he_i)-2f(x))$$
 	and $f$ is any function on $h\mathbb{Z}^d$. We call such a function a grid function. Thus we have, for $x\in h\mathbb{Z}^d$,
    \begin{align*}
     L_hu(x)&=\frac1{h^2}\sum_{i=1}^d (\Delta_hu(x+he_i)+\Delta_hu(x-he_i)-2\Delta_hu(x))\\
     &=\frac1{h^4}\left[ \sum_{i=1}^d \sum_{j=1}^d\left\lbrace u(x+h(e_i+e_j))+u(x-h(e_i+e_j))+u(x+h(e_i-e_j))\right.\right. \\
    &\left.\left.+u(x-h(e_i-e_j))\right\rbrace-4d\sum_{i=1}^d \left\lbrace u(x+he_i)+u(x-he_i)\right\rbrace +4d^2u(x)\right] .
    \end{align*}

 	Let $V_h$ be the set of grid points in $\overline V$ i.e. 
 	$V_h=\overline V\cap h\mathbb{Z}^d$. We say that $\xi$ is an interior 
 	grid point in $V_h$ or $\xi\in R_h$ if for every 
 	$i, \,j$, the points $\xi\pm h(e_i\pm e_j),\,\xi\pm he_i$ are all in $V_h$. 
 	We denote $B_h$ to be $V_h\setminus R_h$. We will denote by $\mathcal{D}_h$ the set of grid functions vanishing outside $R_h$. 
 	For a grid function $f$ we define $R_hf\in\mathcal{D}_h$ by
 	\begin{equation}
\label{eq:R_h}
 	R_hf(\xi)=\begin{cases}
 	f(\xi) &  \xi\in R_h\\
 	0 & \xi\notin R_h.
 	\end{cases}
 	\end{equation}
 	
 	In \cite{thomee} it is crucially used that the discrete approximation of the elliptic operator is consistent. In our case it is easy to see this using Taylor's expansion.
	 	\begin{lemma}\label{lemma:one}
 		The operator $L_h$ is consistent with the operator $L$, that is, if $W$ is a neighborhood of the origin in $\mathbb{R}^d$ and $u\in C^4(W)$ then $$L_hu(0)=Lu(0)+o(1)\,\,\, \text{as } h\to 0.$$
 	\end{lemma}

 	We will divide $R_h$ further into $R^*_h$ and $B^*_h$ where $R^*_h$ is the set 
 	of $\xi$ in $R_h$ such that for every $i,\, j$, 
 	the points $\xi\pm h(e_i\pm e_j),\,\xi\pm he_i$ are all 
 	in $R_h$ and $B^*_h$ is the set of remaining points in $R_h$. 
 	Thus we have $$V_h=B_h\cup R_h=B_h\cup B^*_h\cup R^*_h.$$
  We say that the domain $V$ has property $\mathcal{B}^*_2$ if 
  	there is a natural number $K$ such that for all sufficiently small $h$, 
  	the following is valid: consider for any $\xi\in B^*_h$ all half-rays through $\xi$. 
  	At least one of them contains within the distance $Kh$ from $\xi$ two consecutive 
  	grid-points in $B_h$. 
  	
  	The following Proposition shows that if the boundary of the domain is regular enough then the property $\mathcal{B}^*_2$ is true. Namely, recall the  uniform exterior ball condition (UEBC) for a domain $V$, which states that there exists $\delta>0$ such that for any $z\in\partial V$ there is a ball $B_\delta(c)$ of radius $\delta$ with center at some point $c$ satisfying $\overline{B_\delta(c)}\cap \overline V = \{z\}$ \cite[page 27]{GT79}. We show that the UEBC is a sufficient condition for $\mathcal B_2^*$ to hold. In particular, any domain with $C^2$ boundary satisfy the UEBC and hence possesses $\mathcal B_2^*$.
	
	\begin{proposition} \label{prop:B2star}
If a bounded domain $V$ satisfies the UEBC then the property $\mathcal B^*_2$ holds.
\end{proposition}
Since the proof of this result is purely geometric and combinatorial in nature we discuss it in Appendix~\ref{appendix:B}. We would like to remark that property $\mathcal B_2^*$ is a crucial requirement in the proof of Theorem~\ref{thm:Thomee4.2}. In fact, it allows us to use Thom\'ee's result \cite[Lemma~3.4]{thomee} which compares the standard discrete Sobolev norm with a modified Sobolev norm weighted on boundary points.  
 	
 	%
%
 	We now define the finite difference analogue of the Dirichlet's problem \eqref{eqa:continuum}. For given $h$, we look for a function $u(\xi)$ defined on $V_h$ such that 
 	\begin{align}
 	L_hu_h(\xi)=f(\xi), \quad \xi\in R_h \label{eq:discrete}
 	\end{align} 
and
 	\begin{align}
 	u_h(\xi)=0 ,\quad\xi\in B_h. \label{eq:discrete boundary}
 	\end{align}
It follows from Lemma \ref{lemma:one} and Theorem 5.1 of \cite{thomee} that the finite difference Dirichlet problem \eqref{eq:discrete} and \eqref{eq:discrete boundary} has exactly one solution for arbitrary $f$.
%
Recall also the norm $\|f\|_{h,\,grid}^2:=h^d\sum_{\xi\in h\Z^d}f(\xi)^2.$ Before we prove the approximation theorem, 
 	let us cite two results from~\cite{thomee} (stated, in the original article, in a slightly more general way).
 	\begin{lemma}[{\citet[Lemma~3.1]{thomee}}]\label{lem:Thomee3.1}
 	 There are constants $C>0$ independent of $f$ and $h$ such that
 	 \[
 	  \|f\|_{h,\,grid}\le C\|D_j f\|_{h,\,grid},\quad j=1,\,\ldots,\,d
 	 \]
 	 and 
 	 \[
 	  \|f\|_{h,\,grid}\le C\|f\|_{h,\,2}:=\left(\sum_{|\beta|\le 2}\|D^\beta f\|_{h,\,grid}^2\right)^{1/2}
 	 \]
   for any grid function $f$ vanishing outside $R_h$, where 
    $$D_jf(x) := \frac1h(f(x+he_j)-f(x)),\,\, j=1,\ldots,d$$ and
    $$D^\beta f:=D_1^{\beta_1}\cdots D_d^{\beta_d}f,\; \beta=(\beta_1,\ldots,\beta_d),\;  \beta_i \ge 0.$$
 	\end{lemma}
For the next result we need the definition of the operator $L_{h,2}$ from \cite{thomee} as follows:
 		$$L_{h,2} f(x) = \begin{cases}
 		L_h f(x)  & x\in R_h^\ast\\
 		h^2 L_h f(x)  & x\in B_h^\ast\\
 		0 & x\notin R_h.
 		\end{cases}$$

 	\begin{theorem}[{\citet[Theorem~4.2]{thomee}}]\label{thm:Thomee4.2}
 	  There exists a constant $C>0$ such that for 
 	    all grid functions $f$ vanishing outside $R_h$
 	    \[
 	     \|f\|_{h,\,2}\le C\|L_{h,\,2}f\|_{h,\,grid},
 	    \]
where $C$ is independent of $h$ as well. 
 	\end{theorem}

We have now all the ingredients to show the following.
 	\begin{theorem}\label{thm:one}
 		Let $u\in {C}^5(\overline{V})$ be the solution of the Dirichlet's problem~\ref{eqa:continuum} 
 		and $u_h$ be the solution of the discrete problem ~\eqref{eq:discrete}-\eqref{eq:discrete boundary}. If $e_h:=u-u_h$ then we have for all sufficiently small $h$
 		$$\lVert R_he_h\rVert_{h,\,grid}^2\le C\left[M_5^2h^2 + h(M_5^2h^6+ M_2^2)\right]$$
		where $M_k=\sum_{\lvert\alpha\rvert\le k}\sup_{x\in V}\lvert D^\alpha u(x)\rvert$.
 	\end{theorem}
 
 \begin{proof}
 		We denote all constants by $C$ and they do not depend on 
 		$u,\, f$. Using Taylor's expansion we have for all $x\in R_h$ and for small $h$
 		$$L_hu(x)=Lu(x)+h^{-4}\mathcal R_5(x)$$
 		where $\lvert \mathcal R_5(x)\rvert\leq CM_5h^5$. We obtain for $\xi\in R_h$,
 		\begin{align*}
 		L_he_h(\xi)&=L_hu(\xi)-L_hu_h(\xi)\\
 		&=Lu(\xi)+h^{-4}\mathcal R_5(\xi)-L_hu_h(\xi)=h^{-4}\mathcal R_5(\xi).
 		\end{align*}
 		 		For $\xi\in R^*_h$ we have
 		\begin{align*}
 		L_{h,2}R_he_h(\xi)&= L_hR_he_h(\xi)=L_he_h(\xi)=h^{-4}\mathcal R_5(\xi).
 		\end{align*}
 		For $\xi\in B^*_h$ at least one among $\xi\pm h(e_i\pm e_j),\,\xi\pm he_i$ is in $B_h$. For any $\eta \in B_h\setminus \partial V$ we consider a point $b(\eta)$ on $\partial V$ of minimal distance to $\eta$. Note that this distance is at most $2h$. Now using Taylor expansion and the fact that the value of $u$ and all its first order derivatives are zero at $b(\eta)$ one sees that $$u(\eta)=u_h(\eta)+\mathcal R_2(\eta)$$ where $\lvert \mathcal R_2(\eta)\rvert\leq CM_2h^2$. For $\xi\in B_h^\ast$ denote by $$S_{i,j}(\xi)= \{ \eta: \eta\in B_h\setminus (B_h\cap \partial V) \cap \{ \xi\pm h e_i, \xi\pm h(e_i\pm e_j)\}\}.$$ Therefore, for $\xi\in B_h^\ast$,
 		\begin{align*}
 		L_{h,2} R_he_h(\xi)&= h^2 L_h R_he_h(\xi)\\
 		&=h^2 \left\{L_he_h(\xi)- h^{-4} \sum_{i,j=1}^d \sum_{\eta\in S_{i,j}(\xi)} C(\eta) e_h(\eta)\right\}\\
 		&= h^{-2}\mathcal R_5(\xi) + h^{-2} C\mathcal R^{'}_2(\xi)
 		\end{align*}
 	    where $C(\eta)$ is a constant depending on $\eta$ and $\lvert \mathcal R^{'}_2(\xi)\rvert\leq CM_2h^2$. Hence
 	    \begin{align*}
 	    \|L_{h,2} R_he_h\|_{h,grid}^2 &= h^d\sum_{x\in R_h} (L_{h,2} R_he_h(x))^2\\
 	    &= h^d \left[ \sum_{x\in R_h^\ast} (L_{h,2} R_he_h(x))^2+ \sum_{x\in B_h^\ast} (L_{h,2} R_he_h(x))^2 \right]\\
 	    &= h^d \left[ \sum_{x\in R_h^\ast} (h^{-4}\mathcal R_5(x))^2 +\sum_{x\in B_h^\ast} (h^{-2}\mathcal R_5(x) + h^{-2} C\mathcal R^{'}_2(x) )^2\right]\\
 	    &\le  h^d \left[ \sum_{x\in R_h^\ast} CM_5^2h^2 +\sum_{x\in B_h^\ast} (CM_5^2h^6+ CM_2^2)\right]\\
 	    &\le C\left[M_5^2h^2 + h(M_5^2h^6+ M_2^2)\right]
 	    \end{align*}
 	    where in the last inequality we have used that the number of points in $B_h^\ast$ is $O(h^{-(d-1)})$ following from \citet[Lemma 5.4]{penrose_rgg} and the assumption of a $C^2$ boundary. 
 	    Finally to complete our proof we obtain
 	    \begin{align}\label{eq:errorbound}
 	    \|R_he_h\|_{h,\,grid}^2 \le C\left[M_5^2h^2 + h(M_5^2h^6+ M_2^2)\right]
 	    \end{align}
 	    using Lemma~\ref{lem:Thomee3.1} and Theorem~\ref{thm:Thomee4.2}. This concludes the proof.  
 \end{proof}
 
 \section{}\label{appendix:B}
 Now we provide a proof of Proposition~\ref{prop:B2star}.  
 
\begin{proof}[Proof of Proposition~\ref{prop:B2star}]
If $d=1$ then it is easy to see from the definition that $\mathcal B_2^*$ holds. So we assume $d\ge 2$. For any $y\in V_h$ we denote by $N(y)$ the neighbourhood of $y$, that is, $$N(y):=\{ y\pm he_i, y\pm he_i \pm he_j: 1\le i,j\le d\}.$$
	We consider in fact a second-nearest neighbourhood in the graph distance, due to the interaction of the discrete bilaplacian and Thom\'ee's definition of neighbour. Let us now recall the definitions:
	\begin{align*}
	&V_h=\overline V\cap h\Z^d,\\
	&R_h= \{x\in V_h : N(x)\subseteq V_h\},\\
	&B_h=V_h\setminus R_h,\\
	&R_h^*=\{x\in R_h : N(x)\subseteq R_h\},\\
	&B_h^*=R_h\setminus R_h^*.
	\end{align*}
	Thus $V_h=B_h\cup B_h^* \cup R_h^*$.
	We want to show that for sufficiently small $h$ the following holds: for any $x\in B_h^*$ there exists $i\in \{1,\ldots, d\}$ such that any two consecutive points of either $\{x+he_i,\, x+2he_i,\, x+3he_i,\, x+4he_i\}$ or $\{x-he_i,\, x-2he_i,\, x-3he_i,\, x-4he_i\}$ belong to $B_h$. The proof is done on a case-by-case basis. We prove the existence of two consecutive points by broadly considering the following two possibilities:
	\begin{itemize}[leftmargin=*]
		\item  suppose $x\in B_h^*$ is such that $\mathrm{dist}(x,B_h)=1$, then we get an $i_0\in\{1,\,\ldots,\,d\}$ so that either $x+he_{i_0}$, $x+2he_{i_0} \in B_h$ or $x-he_{i_0}$, $x-2he_{i_0} \in B_h$.
		
		\item  Now suppose $x\in B_h^*$ is such that $\mathrm{dist}(x,B_h)=2$. In this case if $\{ x\pm 2he_i: 1\le i \le d\} \cap B_h$ is non-empty then we get an $i_0\in\{1,\,\ldots,\,d\}$ so that either $x+2he_{i_0}, x+3he_{i_0} \in B_h$ or $x-2he_{i_0}, x-3he_{i_0} \in B_h$. 
		Otherwise, $\{ x\pm 2he_i: 1\le i \le d\} \cap B_h$ is empty and $\{ x\pm he_i \pm he_j: 1\le i,j \le d, i\neq j\} \cap B_h$ is non-empty. And then we extract an $i_0$ so that either $x+3he_{i_0}, x+4he_{i_0} \in B_h$ or $x-3he_{i_0}, x-4he_{i_0} \in B_h$.
	\end{itemize} 
	In the process of obtaining these suitable points, we rule out some of the cases which do not arise due to the regularity of the boundary.            
	
	Fix $x\in B_h^*$. Then $N(x)\subset V_h$ and $N(x)\cap B_h \neq \emptyset$.
	\begin{enumerate}[ref=\arabic*,label*=\arabic*.,leftmargin=*]
		\item\label{item:case1} Suppose $\{ x\pm he_i : 1\le i\le d\} \cap B_h \neq \emptyset$. We assume for simplicity that $x+he_1 \in B_h$ as the argument will be similar for other directions. If $x+2he_1\in B_h$, then there is nothing to prove. More elaborate is the case when $x+2he_1\in R_h$. Then we have 
		\begin{align*}
		&N(x)
		\subseteq \overline V,\\
		&N(x+he_1)
		\nsubseteq \overline V,\\
		&N(x+2he_1)
		\subseteq \overline V.
		\end{align*}
		Observe that from the preceding inclusions we must have 
		\begin{equation}\label{eq:difficult}
		\{ x+he_1\pm he_i \pm he_j: 2\le i,j\le d\} \nsubseteq \overline V.
		\end{equation}
		We now partition this set into 2 subsets and argue separately. 
		\begin{enumerate}[label*=\arabic*.]
			\item Suppose $\{ x+he_1\pm 2he_i : 2\le i\le d\} \nsubseteq \overline V$. Let us assume that $x+he_1+2he_2 \notin \overline V$. Then by definition of $B_h$ we have $x+he_2,\, x+2he_2 \in B_h$ and we are done. Similar is the case for other points.
			
			\item\label{item:case1b} We are left with the situation where $\{ x+he_1\pm 2he_i : 2\le i\le d\} \subset \overline V$ and $\{ x+he_1\pm he_i \pm he_j: 2\le i,j\le d, i\neq j\} \nsubseteq \overline V$. Note that this situation is not possible in $d=2$ and hence from now we consider $d\ge 3$ for this subcase. 
			
			Again we continue with a particular choice $x+he_1+he_2+he_3 \notin \overline V$. The other occurrences can be handled similarly. Note that with this choice we have $x+he_2,\, x+he_3\in B_h$. So if at least one between $x+2he_2$ and $ x+2he_3$ belongs to $B_h$ then we are done. Otherwise we have the following situation: 
			\begin{align*}
			\{x,\, x+2he_1,\, x+2he_2,\, x+2he_3\}\subset  R_h,\\
			\{x+he_1,\, x+he_2,\, x+he_3 \}\subset B_h,\\
			\{ x+he_1\pm 2he_i : 2\le i\le d\} \subseteq \overline V
			\end{align*}
			and $x+he_1+he_2+he_3 \notin \overline V$. Note here that the point $x+he_1+he_2+he_3$, which is at graph distance $3$ from $x$, is not in $\overline V$. However its nearby points $\{x+2he_1+he_2+he_3,\,x+he_2+he_3,\,x+he_1+2he_2+he_3,\, x+he_1+he_3,\, x+he_1+he_2+2he_3,\, x+he_1+he_2\}$ stay inside $\overline V$. We show that such a situation cannot happen due to the UEBC. 
			Indeed, since the domain satisfies UEBC, we can find for small $h$ a ball $B_\delta(c)$ for some $c\in\R^d$ such that $x+he_1+he_2+he_3 \in B_\delta(c)$ and $\overline{B_\delta(c)}\cap \overline V = \{y\}$ for some $y\in\partial V$. Clearly, if $x=(x_1,\,\ldots,\,x_d)$ and $c=(c_1,\,\ldots,\,c_d)$ then
			\begin{equation}\label{eq1}
			\sum_{i=1}^d (c_i-x_i)^2 > \delta^2
			\end{equation}
			and 
			\begin{equation}\label{eq2}
			\sum_{i=1}^3 (c_i-x_i-h)^2 + \sum_{i=4}^d (c_i-x_i)^2 < \delta^2.
			\end{equation}
			Since $x+2he_1+he_2+he_3,\,x+he_2+he_3 \in\overline V$ we have
			\begin{equation}\label{eq3}
			(c_1-x_1-2h)^2+(c_2-x_2-h)^2 +(c_3-x_3-h)^2 + \sum_{i=4}^d (c_i-x_i)^2 \ge \delta^2,
			\end{equation}
			\begin{equation}\label{eq4}
			(c_1-x_1)^2+(c_2-x_2-h)^2 +(c_3-x_3-h)^2 + \sum_{i=4}^d (c_i-x_i)^2 \ge \delta^2.
			\end{equation}
			Now subtracting \eqref{eq3}, respectively~\eqref{eq4}, from ~\eqref{eq2} we get, respectively,
			\begin{align*}
			&(2c_1-2x_1-3h) h \le 0 ,\\
			&(2c_1-2x_1-h)(-h) \le 0 .
			\end{align*}
			Hence
			\begin{equation}\label{eq5}
			(c_1-x_1)^2 \le \frac{9h^2}{4}.
			\end{equation}
			Similarly using the points $x+he_1+2he_2+he_3,\, x+h_1+h_3,\, x+he_1+he_2+2he_3,\, x+he_1+he_2$ in $\overline V$ we obtain
			\begin{equation}\label{eq6}
			(c_2-x_2)^2 \le \frac{9h^2}{4},
			\end{equation}
			\begin{equation}\label{eq7}
			(c_3-x_3)^2 \le \frac{9h^2}{4}.
			\end{equation}
			We now observe that 
			\begin{equation}\label{eq:bad_weather}
			x+he_1\pm he_4\in N(x)\subseteq \overline V.
			\end{equation}
			Consequently
			\begin{align}
			(c_1-x_1-h)^2&+(c_2-x_2)^2 +(c_3-x_3)^2 + (c_4-x_4-h)^2 \nonumber\\
			&+\sum_{i=5}^d (c_i-x_i)^2 \ge \delta^2\label{eq8}
			\end{align}
			and 
			\begin{align}
			(c_1-x_1-h)^2 &+(c_2-x_2)^2 +(c_3-x_3)^2 + (c_4-x_4+h)^2 \nonumber\\
			&+\sum_{i=5}^d (c_i-x_i)^2 \ge \delta^2.\label{eq9}
			\end{align}
			Subtracting ~\eqref{eq8} from ~\eqref{eq2} we derive, after a few simple manipulations,
			\begin{align*}
			(c_4-x_4) \le \frac{11h}{4}.
			\end{align*}
			Similarly subtracting~\eqref{eq9} from~\eqref{eq2} we obtain
			\begin{align*}
			(c_4-x_4) \ge -\frac{11h}{4}.
			\end{align*}
			Thus
			\begin{equation*}
			(c_4-x_4)^2 \le \frac{121 h^2}{16}.
			\end{equation*}
			Re-running the above argument considering $x+he_1\pm he_i\in \overline V$, $5\le i\le d$, in place of $x+he_1\pm he_4$ in~\eqref{eq:bad_weather}, and using equations similar to~\eqref{eq8} and~\eqref{eq9} we obtain all in all that
			\begin{equation}\label{eq10}
			(c_i-x_i)^2 \le \frac{121 h^2}{16},\quad i=4,\ldots, d.
			\end{equation}
			Finally we observe that, for small enough $h$, ~\eqref{eq6},~\eqref{eq7} and~\eqref{eq10} together contradict ~\eqref{eq1}. This completes Case ~\ref{item:case1}.
		\end{enumerate}
		\item\label{item:case2} For this case we have $\{ x\pm he_i : 1\le i\le d\}\cap B_h = \emptyset$ but $\{ x\pm he_i \pm he_j: 1\le i,j\le d\} \cap B_h \neq \emptyset$. Here also we consider two subcases.
		\begin{enumerate}[ref=\arabic*,label*=\arabic*.,leftmargin=*]
			\item First we consider the subcase when $\{ x\pm 2he_i : 1\le i \le d\} \cap B_h \neq \emptyset$. For simplicity we continue with a particular choice $x+2he_1\in B_h$. In this case if $x+3he_1\in B_h$ then we are done. So we assume $x+3he_1\in R_h$. Observe that
			\begin{align*}
			&N(x+2he_1) \nsubseteq \overline V\\
			&N(x\pm he_i), \, N(x+3he_1) \subseteq \overline V
			\end{align*}
			which imply that we must have 
			\begin{equation}\label{eq:21}
			\{ x+2he_1\pm he_i \pm he_j: 1 < i,j\le d\} \nsubseteq \overline V.
			\end{equation} 
			
			We consider two different situations.
			
			\begin{enumerate}[labelindent=0pt,ref=\arabic*,label*=\arabic*.,leftmargin=*]
				\item\label{item:case2a1} Let us first consider the situation when $\{ x+2he_1\pm 2he_i: 1 < i\le d\} \nsubseteq \overline V.$ In particular we consider without loss of generality $x+2he_1+2he_2\notin \overline V$. Note that this implies $x+2he_2 \in B_h$. So if $x+3he_2 \in B_h$ then we are done. Otherwise we have $x+3he_2 \in R_h$. But in this case we see that $x+2he_1+2he_2\notin \overline V$ and its nearby points $\{x+3he_1+2he_2, x+he_1+2he_2, x+2he_1+3he_2, x+2he_1+he_2, x+2he_1+he_2\pm he_i : 3\le i\le d\}$ stay inside $\overline V$. It can be shown that this case is impossible by UEBC with a similar argument as in Case~\ref{item:case1b}
				
				\item\label{item:case2a2} We now consider the other situation (note the such a situation does not appear in $d=2$) when $\{ x+2he_1\pm 2he_i: 1 < i\le d\} \subseteq \overline V$. So using~\eqref{eq:21} without loss of generality we choose a particular element, say $x+2he_1+he_2+he_3 \notin\overline V$. One can show that this situation is not possible for small enough $h$ by arguments similar to Case~\ref{item:case1b} with the observation $\{x+3he_1+he_2+he_3, x+he_1+he_2+he_3, x+2he_2+he_3, x+he_3, x+he_2 + 2he_3, x+he_2, x+h_2+h_3\pm he_i: 4\le i\le d\} \subseteq \overline V$.
			\end{enumerate}
			\item\label{item:case2b} We are left with the subcase when 
			\begin{align}
			&\{ x\pm 2he_i : 1\le i \le d\} \cap B_h = \emptyset, \nonumber\\
			&\{ x\pm he_i \pm he_j: 1\le i,j\le d, i\neq j\} \cap B_h \neq \emptyset.\label{eq:22}
			\end{align}
			Now consider points which are of the form $\{ x\pm 3he_i : 1\le i \le d\}$ and  depending on whether they have non-empty intersection with $B_h$ one can split the argument into two further cases. We use points of the above form as their neighbourhoods contain points which are at graph distance $5$ from $x$ in certain directions.
			\begin{enumerate}[label*=\arabic*.,leftmargin=*]
				\item 
				First we consider the case when $\{ x\pm 3he_i : 1\le i \le d\} \cap B_h \neq \emptyset$. If say, $x+3he_1\in B_h$ then it must be that $x+4he_1\in B_h$ too. Indeed, were this not true one would have 
				\begin{align*}
				& N(x+3he_1)\nsubseteq \overline V,\\
				&N(x+4he_1)\subseteq \overline V,\\
				& N(x+2he_1)\subseteq \overline V.
				\end{align*}
				From these equations we observe that one would have $\{x+3he_1\pm he_i\pm he_j: 1< i,j \le d\}\nsubseteq \overline V$. Now this would give rise to a contradiction by similar argument used in Case~\ref{item:case1b}
				\item We now focus on the case when  
				\begin{equation}\label{eq:23}
				\{ x\pm 3he_i : 1\le i \le d\} \cap B_h = \emptyset.
				\end{equation}
				
				We show that this situation can not arise. To keep the argument simple,  using~\eqref{eq:22}, we assume without loss of generality $x+he_1+he_2\in B_h$. Then 
				$$\{x+he_1+he_2\pm he_i,\, x+he_1+he_2\pm he_i\pm he_j: 1\le i,j \le d\} \nsubseteq \overline V.$$
				Since we are in Case~\ref{item:case2} and \eqref{eq:22}-\eqref{eq:23} hold we have $$N(x\pm he_i), \,N(x\pm 2he_i),\, N(x\pm 3he_i) \subseteq \overline V\text{ for all $i$},$$
				so it must be that 
				\begin{equation}\label{eq:24}
				\{x+he_1+he_2\pm he_i\pm he_j: 3\le i,j \le d, i\neq j\} \nsubseteq \overline V.
				\end{equation}
				Notice that such a situation cannot arise in $d=3$ and  hence we concentrate on $d\ge 4$. To analyse the situation arising out of~\eqref{eq:24}, we suppose $$x+he_1+he_2+he_3+he_4 \notin \overline V.$$ 
				Note that here we cannot follow the steps of Case~\ref{item:case1b} because we do not know if any of the points $x+2he_1+he_2+he_3+he_4,\,x+he_1+2he_2+he_3+he_4,\, x+he_1+he_2+2he_3+he_4, \,x+he_1+he_2+he_3+2he_4$ are in $\overline V$. So we argue in a slightly different way. 
				
				By UEBC for $h$ small enough we can find a ball $B_\delta(c)$ for some $c\in\R^d$ such that $x+he_1+he_2+he_3+he_4 \in B_\delta(c)$ and $\overline{B_\delta(c)}\cap \overline V = \{y\}$ for some $y\in\partial V$. Clearly, if $x=(x_1,\,\ldots,\,x_d)$ and $c=(c_1,\,\ldots,\,c_d)$ then
				\begin{equation}\label{eq11}
				\sum_{i=1}^d (c_i-x_i)^2 > \delta^2
				\end{equation}
				and 
				\begin{equation}\label{eq12}
				\sum_{i=1}^4 (c_i-x_i-h)^2 + \sum_{i=5}^d (c_i-x_i)^2 < \delta^2.
				\end{equation}
				Also $x+he_2+he_3+he_4\in\overline V$ gives
				\begin{equation}\label{eq13}
				(c_1-x_1)^2 + \sum_{i=2}^4 (c_i-x_i-h)^2 + \sum_{i=5}^d (c_i-x_i)^2 \ge \delta^2.
				\end{equation}
				Subtracting ~\eqref{eq13} from ~\eqref{eq12} we get
				\begin{align*}
				c_1-x_1 \ge \frac{h}2.
				\end{align*}
				Similarly we obtain
				\begin{align*}
				c_i-x_i \ge \frac{h}2,\quad i=2,\,3,\,4.
				\end{align*}
				Now we impose a condition on the maximum value of $\{c_i-x_i: i=1,\,2,\,3,\,4\}$ and see that when it is bounded by a factor of $h$ one gets a contradiction. Let $c_k-x_k=\max\{c_i-x_i: i=1,\,2,\,3,\,4\}$. First suppose $c_k-x_k \le 7h/2$. Then we have
				\begin{align*}
				(c_i-x_i)^2 \le \frac{49h^2}4,\quad i=1,\,2,\,3,\,4.
				\end{align*}
				Now using $\{x+he_1+he_2\pm he_j: 5\le j\le d\}\subseteq \overline V$ we deduce 
				\begin{align*}
				(c_j-x_j)^2 \le Ch^2,\quad 5\le j \le d.
				\end{align*}
				where $C$ is a constant depending on $d$. Thus we obtain
				\begin{align*}
				\sum_{i=1}^d(c_i-x_i)^2 \le Ch^2
				\end{align*}
				for some constant $C$. This contradicts ~\eqref{eq11} for small enough $h$. Now suppose we are not in the above situation, that is, $c_k-x_k > 7h/2$. For simplicity let $k=4$. Then we find a contradiction by observing that the point $x+he_2+he_3+3he_4$ can not lie in $\overline V$. Indeed, we have
				\begin{align*}
				&(c_1-x_1)^2+(c_2-x_2-h)^2+(c_3-x_3-h)^2+(c_4-x_4-3h)^2 + \sum_{i=5}^d (c_i-x_i)^2\\
				&\qquad\qquad -\sum_{i=1}^4 (c_i-x_i-h)^2 - \sum_{i=5}^d (c_i-x_i)^2\\
				&=(c_1-x_1)^2-(c_1-x_1-h)^2 + (c_4-x_4-3h)^2-(c_4-x_4-h)^2\\
				&= -h[2(c_4-x_4)-7h +2((c_4-x_4)-(c_1-x_1))] <0.
				\end{align*}
				Thus
				\begin{align*}
				&(c_1-x_1)^2+(c_2-x_2-h)^2+(c_3-x_3-h)^2+(c_4-x_4-3h)^2 + \sum_{i=5}^d (c_i-x_i)^2 < \delta^2.
				\end{align*}
				This implies that $x+he_2+he_3+3he_4 \in B_\delta(c)$ which is impossible as $x+he_2+he_3+3he_4\in N(x+3he_4)\subseteq \overline V$.
				This completes the proof.\qedhere
			\end{enumerate}
		\end{enumerate}
	\end{enumerate}
\end{proof}

\bibliographystyle{abbrvnat}
\bibliography{biblio}

\end{document}